\documentclass[twoside,leqno,10.5pt]{article} %leqno is the option to put formula numbers on the left side
\usepackage{xcolor}         
\usepackage{fancyhdr}
\usepackage{amsmath, amssymb}  
\usepackage{amsthm} % environment option 
\usepackage{amscd}     
\usepackage{multicol,lineno}
\usepackage{amsmath, amssymb}    
\usepackage{amsthm} %theorem environment option
\usepackage{ascmac} 
\usepackage{amscd}        
\usepackage{color}     
\usepackage[normalem]{ulem}

\newcommand{\ev}{{\rm Eval}}  
\newcommand{\inte}{{\rm Prim}} 
\newcommand{\deri}{{\rm Deri}} 
\newcommand{\id}{{\rm Id}}   
\newcommand{\card}{{\rm Card}}  

\newcommand{\const}{{rm\log(2)\rule{0mm}{4mm}+r\left(\log(rm+1)+rm\log\left(\rule{0mm}{3.5mm}\frac{rm+1}{rm}\right)\right)}} 
\theoremstyle{plain} %text of this environment is typesetted in italics
\newtheorem{theorem}{\indent\sc Theorem}%[section]
\newtheorem{lemma}{\indent\sc Lemma}
\newtheorem{corollary}{\indent\sc Corollary}
\newtheorem{proposition}{\indent\sc Proposition}

\theoremstyle{definition} %text of this environment is typesetted in roman letters
\newtheorem{definition}{\indent\sc Definition}
\newtheorem{facts}{\indent\sc Fact}
\newtheorem{remark}{\indent\sc Remark}
\newtheorem{example}{\indent\sc Example}
\newtheorem{notation}{\indent\sc Notation} 

\newcommand{\C}{\mathbb{C}}
\newcommand{\R}{\mathbb{R}}
\newcommand{\ru}{{\R}} 
\newcommand{\Q}{\mathbb{Q}}
\newcommand{\sign}{{\rm sign}}

\newcommand{\Z}{\mathbb{Z}} 
\newcommand{\qu}{{\Q}}
\newcommand{\zu}{{\Z}}
\newcommand{\pu}{{\mathbb{P}}}
\newcommand{\N}{\mathbb{N}}
\newcommand{\Li}{\rm{Li}}

\newcommand{\cd}{\cdots}

\def\cd{\cdots}

\def\2{I\hspace{-.1em}I}
\def\Li{\textrm{Li}}

\title{Linear Forms in Polylogarithms}
\author{\textsc{Sinnou David}, \textsc{Noriko Hirata-Kohno} and \textsc{Makoto Kawashima}}
\date{}

\begin{document}

\maketitle

\begin{abstract}
Let $r, \,m$ be positive integers.
Let$x$ be a rational number with $0 \le x <1$. Consider
$\Phi_s(x,z) =\displaystyle\sum_{k=0}^{\infty}\frac{z^{k+1}}{{(k+x+1)}^s}$ the $s$-th Lerch function with $s=1, 2, \cd, r$.
When $x=0$, this is a polylogarithmic function.
Let $\alpha_1, \cdots, \alpha_m$
be pairwise distinct algebraic numbers of \emph{arbitrary degree} over the rational number field, 
with $0<|\alpha_j|<1 \,\,\,(1\leq j \leq m)$.
In this article, we show a criterion for the linear 
independence, over {{an}} algebraic number field containing $\Q(\alpha_1, \cdots, \alpha_m)$,  of all the $rm+1$ numbers~:
$\Phi_1(x,\alpha_1)$, \,$\Phi_2(x,\alpha_1), \,$
$\cdots , \ , \Phi_r(x,\alpha_1)$, \,\,$\Phi_1(x,\alpha_2)$, \,$\Phi_2(x,\alpha_2), \,$
$\cdots , \, \Phi_r(x,\alpha_2), \cdots \cdots, \Phi_1(x,\alpha_m)$, 
~$\Phi_2(x,\alpha_m)$, ~$\cdots , \Phi_r(x,\alpha_m)$ and $1$. This is the first result that gives a sufficient condition for
the linear independence of values of the Lerch functions at  \emph{several}  \emph{distinct algebraic points},
\emph{{not necessarily lying in the rational number field nor in quadratic imaginary fields}.}
We give a complete proof with refinements and quantitative statements of the main theorem
announced in \cite{DHK2}.
%}}%%%%%%%%%%%%
\end{abstract}
Mathematics Subject Classification (2020): 11J72 (primary); 11J82, 11J61 (secondary).

\section{Introduction}

We study the linear independence over number fields, {{as well as}} explicit linear independence measures, of
values of the Lerch functions at different  algebraic points. 
 
Following  the pioneering work of Pad\'e  ({confer \cite{Pade1}, \cite{Pade2}}), Pad\'e approximations have evolved as major tool in
Diophantine problems.
They have been extensively {{conducted}} over the past decades, since they naturally provide for sharp approximations and outstanding applications in various researches:
for the study of algebraic functions on Jacobians of irreducible curves by E.~Bombieri, P.~Cohen and U.~Zannier{, see \cite{bomcoza}} and for
Pellian equations over function fields, by Zannier, {\em confer} \cite{Za}. 

The linear independence of values of special functions is our present main focus. In our set up, we go back to previous works
by A.~I.~Galochkin \cite{G1}, %\cite{G2}, 
 G.~V.~Chudnovsky  \cite{ch11}, Y.~Z.~Flicker, \cite{Fli}, K.~V$\ddot{\text{a}}$$\ddot{\text{a}}$n$\ddot{\text{a}}$nen \cite{Va} and W. Zudilin \cite{Z}. Indeed, polylogarithmic functions as well as the Lerch functions fall in the class of $G$-functions in the sense of C.~L.~Siegel \cite{Siegel}. 
Galochkin proved a conditional result of the linear independence, and Chudnovsky showed that the
set of functionally linearly independent $G$-functions automatically satisfies Galochkin's condition.
Further it was generalized by Flicker for the $p$-adic version, by V$\ddot{\text{a}}$$\ddot{\text{a}}$n$\ddot{\text{a}}$nen for simultaneous approximations,
and by Zudilin for the irrationality of values of $G$-functions over the rational number field.
We refer to our previous work \cite{DHK2} for a detailed discussion of the above mentioned results. 
 
However, several issues need to be faced to pursue the Pad\'e program: firstly, a concrete system of Pad\'e approximants must  be constructed.
This can be done either explicitly, or by Siegel's lemma via Dirichlet's box principle. 
Indeed, explicit construction of such systems usually gives sufficiently sharp estimates, but it has been achieved for specific functions only, conversely, Siegel's lemma facilitates tackling larger classes, albeit the cost of estimate. The second route has been taken by 
Galochkin, {{Chudnovsky, Flicker and V$\ddot{\text{a}}$$\ddot{\text{a}}$n$\ddot{\text{a}}$nen ({\it confer loc. cit.}).
{On the other hand, in a closely related question, explicit constructions to obtain a lower bound for the dimension
of the subspace spanned by such values over $\mathbb{Q}$ have been taken by  S.~Fischler-J.~Sprang-Zudilin \cite{FSZ}, leading to very sharp results.  P. Bel and  Sprang showed $p$-adic analogues \cite{Bel}{,} \cite{SprangL} of lower bounds of this kind.}

Finally, 
{\it provided} one proves the injectivity of the
evaluation maps (indeed, the non-vanishing {of a} Wronskian of Hermite  type) associated to the constructed Pad\'e system, 
one derives the irrationality or arithmetical properties of special values in question.
Namely, whenever we succeed in proving the injectivity,  a fully effective treatment of the given Diophantine problems {{can be achieved}}, 
associated to the class of functions.

The linear independence of  $\Li_s(\alpha)$ at one $\alpha\in\Q$, with $1\leq s \leq r$ was proven by E.~M.~Nikisin \cite{N}
and was generalized in \cite{H-I-W}, \cite{H-K-S}, \cite{Ka}.
M.~Hata \cite{Ha}{,} \cite{Ha1993} 
obtained the irrationality of $\Li_2(1/q)$ for $q\geq 7$ or $q\leq -5$.
T.~Rivoal \cite{Ri} and R.~Marcovecchio \cite{marc} succeeded in showing
a lower bound for the dimension of a vector space spanned by polylogarithms over an algebraic number field.
M.~A.~Miladi \cite{Mi} showed
the linear independence of $1, \,\Li_1(1/q)=-\Li_1(1/(1-q))$, $\Li_2(1/q)$ and $\Li_2(1/(1 -q))$ over $\Q$
for an integer $q\geq 11$, giving a complete proof for an announcement by D.~V. and G.~V.~Chudnovsky.
In 2018, 
C.~Viola and Zudilin \cite{VZ1} adapted the method due to G.~Rhin and Viola \cite{RV1}
to establish the linear independence of 
$1, \, \Li_1(\alpha)$, $\Li_2(\alpha)$, $\Li_2(\alpha/(\alpha-1))$ for a larger set of $\alpha\in\Q$.

At distinct algebraic numbers $\alpha_1, \cdots, \alpha_m$ satisfying certain conditions,
Rhin and P.~Toffin \cite{R-T} showed the linear independence
of the natural  logarithm by Pad\'e approximations of {{the second kind}.}
Although their insight was underexplored in the subsequent years,  we undertook here
to generalize their construction to deal with polylogarithms
as well as values of distinct algebraic numbers of the Lerch function to
obtain the linear independence and the irrationality of these values.

More general approach {have} also been extensively studied: Rivoal constructed explicit systems of Pad\'e approximants of {{the first kind}}
for the Lerch functions at one point \cite{Rilerch}.
Yu.~Nesterenko constructed ex\-pli\-cit systems of Pad\'e approximants of generalized hypergeometric functions \cite{Nest} and
developed earlier constructions of P.~L.~Ivankov \cite{Ivan}.

\bigskip

However, the Pad\'e approximants come with a codicil. Even when constructed, one may face a difficulty coming from the 
size of the coefficients of the approximants: they 
sometimes become  too large, for any meaningful result to be derived.
This phenomenon is often unavoidable, however, an essential improvement can be obtained, as is shown in \cite{bomcoza}, \cite{Za}, by deep observations of the objects.
Bombieri ({\it confer} \cite{bom}) circumvented the difficulty 
by essentially not going to the maximal possible order of approximation achieved, thus keeping some room to reduce the height of the coefficients of
the approximants.
This still leaves questions unanswered, {\it confer} \cite{bomhunpoor} for a discussion.
 
As hinted above, one needs to prove that the system of approximants is a non-trivial one. The injectivity of the associated evaluation maps translates into whether 
certain determinants of Hermite type vanish or not.  There {{has been
no systematic recipe for  this most difficult part of the Pad\'e program.}}

\vspace{\baselineskip}

Let  $s$ be a positive integer and $x$ be a rational number with $0 \le x <1$.  
In this article, we show the linear independence of values of the $s$-th Lerch function
defined by (for $p=\infty$ or a rational prime)
$$
{{\Phi_{s,p}}}(x,z)=\Phi_s(x,z) =\sum_{k=0}^{\infty}\frac{z^{k+1}}{{(k+x+1)}^s}~, \,\,\,\,\,\, {{z\in \C_p, \,\,\,  \vert z\vert_p<1}}\,,
$$
at distinct algebraic points with quantitative linear independence measures.

When $x=0$, this function is 
${\rm{Li}}_s(z)=\displaystyle\sum_{k=0}^{\infty}\frac{z^{k+1}}{{(k+1)}^s}~$: the $s$-th polylogarithmic function. 
The $s$-th Lerch function $\Phi_s(x,z)$ satisfies the inhomogenous differential equation\kern3pt
\,\,\,\,{{$\left(z\tfrac{d}{dz}+x\right)\Phi_s(x,z)=\Phi_{s-1}(x,z)$}} and 
the Lerch function is a $G$-function in the sense of Siegel  \cite{Siegel}.

Let $r$ be a positive integer. 
Consider $r$ Lerch functions $\Phi_s(x,z), \,\,1\leq s \leq r$ and let $K$ be an algebraic number field of arbitrary finite degree over $\Q$.
Let  $\alpha_1, \cdots, \alpha_m\in K$ be pairwise distinct, 
being sufficiently close to the origin, {{namely satisfying standard metric conditions in Pad\'e approximation}}, that we will later precise.

In this article, we give a criterion to show
the linear 
independence over $K$ of all the $rm+1$ numbers~: 
$\Phi_1(x,\alpha_1)$, \,$\Phi_2(x,\alpha_1), \,$
$\cdots , \ , \Phi_r(x,\alpha_1)$, \,\,$\Phi_1(x,\alpha_2)$, \,$\Phi_2(x,\alpha_2), \,$
$\cdots , \, \Phi_r(x,\alpha_2), \cdots \cdots, \Phi_1(x,\alpha_m)$, 
~$\Phi_2(x,\alpha_m)$, ~$\cdots , \Phi_r(x,\alpha_m)$ and $1$.

Our approach is inspired by previous
works due to Nikisin \cite{N} and  Rhin-Toffin \cite{R-T}. We construct explicit approximants, thus avoiding Siegel's lemma, but do choose a {\it formal} construction. This has a double advantage: it is possible to write down explicit {{formul\ae}} at every single step of the proof, and additionally it is also feasible to {perform} estimates using the formal definition of the approximants. Indeed, it makes analytic {{method}} easier to handle, {since one can replace  sophisticated estimates involving explicit formulae of the approximants, by simple operator norm estimates, see section~\ref{analy}}.  

The second main ingredient is the  key proposition ensuring non-vanishing of the Hermite type determinant that guarantees the linear independence of the system of approximants being chosen.
The proof of this non-vanishing statement itself relies on the formal construction of the approximants which makes factorization easier to {{carry out}}. 
We take advantage of the explicit nature of the construction at the last step of the proof of this proposition~: when all the variables are specialized, one is reduced to an explicit numerical determinant involving integrals of powers of the classical logarithm function.  

\vspace{\baselineskip}
{Recall that for a} rational shift $x$, the associated Lerch functions are $G$-functions~: this ensures that the size of the approximants is small enough. 
For algebraic irrational shifts $x$, however, this becomes {{an obstacle}}~: see Remark 2.2 (ii) below.
{{We also note that Zudilin showed that the functions $\Phi_1(x,z), \cdots, \Phi_r(x,z)$ are algebraically independent over $\C(z)$ ({\it confer} Lemma 3.1 in \cite{Z}).}}

This paper is {{one}} part of a systematic study of these questions~: for a positive integer $d$, let $x_1, \cdots, x_{d}\in \Q\cap [0,1)$ be \emph{pairwise distinct rational shifts}.
Consider the $r_1+r_2+\cdots +r_d$ Lerch functions~:
$\Phi_1(x_1, z)$, $\Phi_2(x_1, z), \ldots ,  \Phi_{r_1}(x_1, z)$, 
$\Phi_1(x_2, z)$, \,$\Phi_2(x_2, z), \ldots ,  \Phi_{r_2}(x_2, z)$, 
$\ldots, \Phi_1(x_d, z)$, $\Phi_2(x_d, z), \ldots   , \Phi_{r_d}(x_d, z)$. 

In the forthcoming article \cite{DHK4}, we will publish a criterion showing 
the linear independence of the values at distinct algebraic numbers
$\alpha_1, \ldots, \alpha_m$ of the different Lerch functions with these distinct shifts, namely
the linear independence of the $m(r_1+r_2+\cdots +r_d)+1$ numbers:~
$\Phi_1(x_1, \alpha_1)$, $\Phi_2(x_1, \alpha_1), \ldots ,  \Phi_{r_1}(x_1, \alpha_1)$, 
$\Phi_1(x_2, \alpha_2)$, $\Phi_2(x_2, \alpha_2), \ldots ,  \Phi_{r_2}(x_2, \alpha_2)$, 
$\ldots, \Phi_1(x_d, \alpha_m)$, \,$\Phi_2(x_d, \alpha_m), \ldots ,  \Phi_{r_d}(x_d, \alpha_m)$
and $1$. 

\medskip

\noindent
{\bf Acknowledgements}

\smallskip

We are deeply grateful to Professor Wadim Zudilin and Professor Raffaele Marcovecchio who kindly conveyed us their comments on our earlier version of the article, helping us to improve it in various aspects.

We sincerely thank the referee of the present article for significant and precise comments.
We also thank Professor Masaru Ito for his support to calculate examples using Mathematica.

This work is
partly supported by JSPS KAKENHI Grant no.~18K03225, and
also by the Research Institute for Mathematical Sciences, an international joint usage
and research center located in Kyoto University.

\section{Notations and Main results} 
We collect some notations which we use throughout this article. Let $\Q$ be the rational number field and  $K$ an algebraic number field of arbitrary degree $[K:\Q] < \infty$.
Let us denote by $\N$ the set of strictly positive integers.
We denote the set of places of $K$  by ${{\mathfrak{M}}}_K$ (by ${\mathfrak{M}}^{\infty}_K$ for infinite places, by ${{\mathfrak{M}}}^{f}_K$
for finite places, respectively).
For $v\in {{\mathfrak{M}}}_K$, we denote the completion of $K$ with respect to $v$ by $K_v$, and 
{{the completion of an algebraic closure of $K_v$ by  $\mathbb C_v$ (resp. for $v\in {\mathfrak{M}}^{\infty}_K$,
for $v\in{{\mathfrak{M}}}^{f}_K$)
}}.

For $v\in {{\mathfrak{M}}}_K$, we define the normalized absolute value $| \cdot |_v$ as follows~:
\begin{align*}
&|p|_v:=p^{-\tfrac{[K_v:\Q_p]}{[K:\Q]}} \ \text{if} \ v\in{{\mathfrak{M}}}^{f}_K \ \text{and} \ v|p\enspace,\\
&|x|_v:=|\iota_v x|^{\tfrac{[K_v:\R]}{[K:\Q]}} \ \text{if} \ v\in {{\mathfrak{M}}}^{\infty}_K\enspace,
\end{align*}
where $p$ is a prime number and $\iota_v$ the embedding $K\hookrightarrow \C$ corresponding to $v$. 
{On $K_v^n$, the norm $\Vert\cdot\Vert_v$ denotes the norm of the supremum.}
Then we have the product formula 
\begin{align*} 
\prod_{v\in {{\mathfrak{M}}}_K} |\xi|_v=1 \ \text{for} \ \xi \in K\setminus\{0\}\enspace.
\end{align*}

\

{Let $\beta\in K$, we define the absolute Weil height of  $\beta$ as 
\begin{align*}
&{{\mathrm{H}}({\beta}):=\prod_{v\in {{\mathfrak{M}}}_K} \max\{ 1,|\beta|_v\}\enspace,}
\end{align*}
Let $m$ be a positive integer and ${\boldsymbol{\beta}:=(\beta_0,{\ldots},\beta_m) \in\pu_m( K)}$.  
We define the absolute Weil height of $\boldsymbol{\beta}$ by
\begin{align*}
&{{\mathrm{H}}(\boldsymbol{\beta}):=\prod_{v\in {{\mathfrak{M}}}_K} \max\{ |\beta_0|_v,\ldots,|\beta_m|_v\}\enspace,}
\end{align*}
{{and logarithmic absolute Weil height by ${\rm{h}}(\boldsymbol{\beta}):={\rm{log}}\, \mathrm{H}(\boldsymbol{\beta})$. 
Let $v\in \mathfrak{M}_K$, then $h_v(\boldsymbol{\beta})=\log\Vert \boldsymbol{\beta}\Vert_v$ where $\Vert\cdot\Vert_v$ is the sup $v$-adic norm. Then we have ${\mathrm{h}}(\boldsymbol{\beta})={\displaystyle{\sum_{v\in \mathfrak{M}_K}}}{\mathrm{h}}_v(\boldsymbol{\beta})$ and for $\beta\in K$, $h(\beta)$ is the height of the point $(1,\beta)\in\pu_1(K)$. {Similarly, $h_{\infty}(y)=\displaystyle{\sum_{v\mid\infty}}h_v(y)$ is the archimedean part of the height and $h_f(y)=\displaystyle{\sum_{v\nmid\infty}}h_v(y)$ is the finite part of the height.}}}}

\vspace{\baselineskip}

For a finite set $S\subset \overline{\Q}$, we define the denominator of $S$ by $${\rm{den}}(S)=\min \{n\in \Z_{>0}\mid n\alpha \ \text{are algebraic integer for all} \ \alpha \in S \}\enspace.$$}}

To state our main results, we prepare further notations.
\bigskip

Let $m, r$ be positive integers and $K$ a number field. 
Let $x\in \Q\cap [0,1)$ and put $b:={\rm{den}}(x)$. Remark that $b=1$ if $x=0$.

Let $\beta\in K\setminus\{0\}$.  Let $\alpha_1,\ldots,\alpha_m\in K\setminus\{0\}$ be pairwise distinct algebraic numbers. 
We put $\boldsymbol{\alpha}:=(\alpha_1,\ldots,\alpha_m)\in (K\setminus\{0\})^m$.

For a place $v\in \mathfrak{M}_K$ and $f(z)={\displaystyle{\sum_{k=0}^{\infty}}}f_k/z^{k+1}\in {{1/z\cdot}}K[[1/z]]$, 
{recall that
$\iota_v:K \hookrightarrow K_v{{\subset \mathbb{C}_v}}$ is the $v$-adic embedding} and write $f_v(z):={\displaystyle{\sum_{k=0}^{\infty}}}\iota_v(f_k)/z^{k+1}\in {{1/z\cdot}}\mathbb{C}_v[[1/z]]$.
We then consider $f_v(z)$ as the $v$-adic analytic function inside its radius of convergence.

{{Finally, for a place $v\in \mathfrak{M}_K$, we denote
{{\begin{align*}
V_{v}(\boldsymbol{\alpha},\beta)&=\log\vert\beta\vert_{v_0}-rm{\mathrm{h}}(\boldsymbol{\alpha},\beta)-{{(rm+1)}}\log\Vert \boldsymbol{\alpha}\Vert_{v_0}{+rm\log\Vert(\boldsymbol{\alpha},\beta)\Vert_{v_0}}\\ 
&-\displaystyle\left[\const\right]-r\log\mu(x)-br^2m\enspace.
\end{align*}
}} 
We write $V(\boldsymbol{\alpha},\beta)=V_{v}(\boldsymbol{\alpha},\beta)$ when no confusion occurs.}}

Our main result is  as follows.
This is a direct consequence of more precise Theorem~\ref{Lerch 2} which will be proven in Section~\ref{analy}\,.
\begin{theorem} \label{Lerch}
{{Let $v_0$ be a place in $\mathfrak{M}_K$}}, either archimedean or non-archimedean, such that\footnote{By the assumption of $V_{v_0}(\boldsymbol{\alpha},\beta)>0$, we have $||\boldsymbol{\alpha}||_{v_0}<|\beta|_{v_0}$.} 
$$V_{v_0} (\boldsymbol{\alpha},\beta)>0\enspace.$$
{{Then the functions $\Phi_s(x,z), \,\,1\leq s \leq r$ converge  around $\alpha_j/\beta \,(1\leq j \leq m)$  in $K_{v_0}$ and \\the $rm+1$ numbers~$:$ 
$$1,\Phi_{1}(x,\alpha_1/\beta),\ldots,\Phi_{r}(x,\alpha_1/\beta),\ldots, \Phi_{1}(x,\alpha_m/\beta),\ldots,\Phi_{r}(x,\alpha_m/\beta)\enspace,$$ are linearly independent over $K$.}}
\end{theorem}

\begin{remark}
\noindent (i) The theorem is valid for a number field $K$ of finite degree over $\Q$ without special assumption, however, we should note that the condition $V(\boldsymbol{\alpha},\beta)>0$ is heavily dependent on the field $K$, {{although the heights are absolute ones and invariant under base change. 
The main part giving positive contribution in $V(\boldsymbol{\alpha},\beta) $ unfortunately tends to zero (fixing $\beta$ and varying the field), whenever the degree of $K$ goes to infinity,  because of our normalization of absolute values. 
Hence the assumption becomes more restrictive when $K$ is larger.}} We may thus understand that $K=\Q(\alpha_1, \cdots, \alpha_m)$.
We remark also that the quantities  $||\cdot ||_{v_0}$ and $  |\cdot|=|\cdot|_{v_0}$ depend on the degree
$[K:\Q]$.

\noindent (ii) When  $x\in \overline{\Q}\setminus \Q$,
let us mention that $\Phi_r(x,\alpha z)$ is not a $G$-function in the sense of Siegel (cf. \cite{Siegel}).
Let $\alpha\in \overline{\Q}\setminus\{0\}$ and $r\in \N$.
First, for $\alpha_1,\ldots,\alpha_{r+1},\beta_1,\ldots,\beta_r\in \overline{\Q}\setminus \Z_{<0}$, define the generalized hypergeometric function by
$${}_{r+1}F_r(\alpha_1,\ldots,\alpha_{r+1},\beta_1,\ldots,\beta_r;z):=\sum_{k=0}^{\infty}\dfrac{(\alpha_1)_k\ldots (\alpha_{r+1})_k}{(\beta_1)_k\ldots(\beta_r)_kk!}z^k\enspace.$$
Then $\Phi_r(x,z)$ can be represented by
$$\Phi_r(x,z)=\dfrac{z}{(x+1)^r}\cdot {}_{r+1}F_r(x+1,\ldots,x+1, 1~;x+2,\ldots,x+2;z)\enspace.$$ 
Using the characterization of non-polynomial hypergeometric $G$-functions obtained by Galochkin \cite[page $6$]{G3} (see also \cite{R}), we {deduce that} the generalized hypergeometric function ${}_{r+1}F_r(x+1,\ldots,x+1, 1~;x+2,\ldots,x+2;z)$ is not a $G$-function.

{\noindent (iii) The present result is slightly better in general {(but not always)} than the one we announced {without proof} in \cite{DHK2} as far as the quantitative aspects are concerned: analytic estimates are better when $r,m$ are large and the dependence in the height of $\boldsymbol{\alpha},\beta$ is also better. We shall quote the precise statement of \cite{DHK2} and compare it to what is obtained here in Section~\ref{analy} after stating Theorem~\ref{Lerch 2}.}
  
\end{remark}

\begin{corollary} \label{Cor4} 
Let $r,m,d$ be positive integers and $\boldsymbol{\alpha}:=(\alpha_1,\ldots,\alpha_m)$ an $m$-tuple of non-zero pairwise distinct rational numbers.
Let $(\beta_M)_{M\in \N}$ be a family of algebraic numbers with $[\Q(\beta_M):\Q]=d$, $v_{0,M}$ be an archimedean place of $\Q(\beta_M)$. We assume 
there exist positive real numbers $C_1,C_2$ with
\begin{align*}
&({\rm{i}}) \ |\beta_M|_{v_{0,M}}\to \infty \ (M\to \infty)\enspace,\\
&({\rm{ii}}) \ |\beta_M|_{v} \le C_1  \ \text{for} \ M\in \N, v\in {{\mathfrak{M}}}^{\infty}_{\Q(\beta_M)} \ \text{and} \ v\neq v_{0,M}\enspace,\\
&({\rm{iii}}) \ {\rm{den}}(\beta_M)\le C_2 \ \text{for} \ M\in \N\enspace.
\end{align*}
Then there exists a sufficiently large positive integer  $M_0:=M_0(d,r,m,\boldsymbol{\alpha},C_1,C_2)$ depending on $d,r,m$, $C_1,C_2$ and $\boldsymbol{\alpha}$ such that, for any $M\ge M_0$,
the $rm+1$ numbers~$:$
{{$$1,\, {\rm{Li}}_{1}(\alpha_1/\beta_M),\ldots, {\rm{Li}}_{r}(\alpha_1/\beta_M),\ldots, {\rm{Li}}_{1}(\alpha_m/\beta_M),\ldots, {\rm{Li}}_{r}(\alpha_m/\beta_M)\enspace,$$}}
are linearly independent over the degree $d$ algebraic number field $\Q(\beta_M)$. 
\end{corollary}
\begin{proof}
Let $V(\boldsymbol{\alpha},\beta_M)$ be the real number defined in Theorem $\ref{Lerch}$ for $x=0$. 
Then we have
{{\begin{align*}
V(\boldsymbol{\alpha},\beta_M)&=\log\vert\beta_M\vert_{v_{0,M}}-rm{\mathrm{h}}(\boldsymbol{\alpha},\beta_M)-{{(rm+1)}}\log\Vert \boldsymbol{\alpha}\Vert_{v_0}+rm\log\Vert(\boldsymbol{\alpha},\beta_M)\Vert_{v_0}\\ 
&-\displaystyle\left[\const\right]-{r^2m}\enspace.
\end{align*}
}}
By Theorem $\ref{Lerch}$, we are reduced to proving the  existence of  a sufficiently large positive integer $M_0=M_0(d,r,m,\boldsymbol{\alpha},C_1,C_2)$ depending on $d,r,m$ and $\boldsymbol{\alpha}$ such that $V(\boldsymbol{\alpha},\beta_M)>0$ for any $M\ge M_0$. 
Using $(i)$, $(ii)$ and $(iii)$, we have 
{{
\begin{align*}
-{\mathrm{h}}(\boldsymbol{\alpha},\beta_M)+\log||(\boldsymbol{\alpha},\beta_M)||_{v_0}\ge -\sum_{\substack{v|\infty \\ v\neq v_{0,M}}}\max({\rm{h}}_{v}(\boldsymbol{\alpha}),{\rm{log}}(C_1))-
\max(\log{\rm{den}}(\boldsymbol{\alpha}),{\rm{log}}(C_2))\enspace,
\end{align*} 
}}
for sufficiently large $M$ and thus
{{{\small{\begin{align*}
V(\boldsymbol{\alpha},\beta_M)&\ge{\rm{log}}|\beta_M|_{v_{0,M}}-\sum_{\substack{v|\infty \\ v\neq v_{0,M}}}\max({\rm{h}}_{v}(\boldsymbol{\alpha}),{\rm{log}}(C_1))-
\max(\log{\rm{den}}(\boldsymbol{\alpha}),{\rm{log}}(C_2))\\
&-{{(rm+1)}}\log\Vert \boldsymbol{\alpha}\Vert_{v_0}-\displaystyle\left[\const\right]-{r^2m}\enspace.
\end{align*}}}}}
By the above inequality, using the assumption $(i)$, we obtain $V(\boldsymbol{\alpha},\beta_M)>0$ for any sufficiently large positive integer  $M$. 
Thus the assertion of Corollary $\ref{Cor4}$ is verified.
\end{proof}
 
\begin{example} \label{ex1}
Let $d$ be a positive integer  and $f(X):=X^d+c_{d-1}X^{d-1}+{{\cdots} }+c_0\in \Q[X]$ with $c_0\neq 0$.
Put $c:={\rm{den}}(c_0,\ldots,c_{d-1})$. We denote the roots of $f(X)$ by $\gamma_1,\ldots,\gamma_d\in \C$. 
Let $p$ be a prime number which is coprime to $c$. For a positive integer  $M$ which is coprime to $p$, we define a polynomial $f^{(p)}_{M}(X)$ by
$$f^{(p)}_M(X):=\left(p+\dfrac{1}{M}\right)X^{d+1}+pc_{d-1}X^{d}+{\cdots} +pc_0X+\dfrac{p}{M}\enspace.$$
Notice that, by the Eisenstein irreducibility criterion of polynomials, the each polynomial $f^{(p)}_{M}(X)$ is irreducible over $\Q[X]$.
By the theorem of Ostrowski  proving  the continuity of the complex roots of a polynomial with respect to the coefficients, there exists a complex root of $f^{(p)}_{M}(X)$ (see \cite[Theorem $1.4$]{M} or \cite[Theorem $6.2$]{C-M}), say $1/\beta_M$, with
$|1/\beta_M|$ is sufficiently close to $0$ and for each conjugate $1/\beta^{(g)}_M$ of $1/\beta_M$, there exists some $1\le i \le d$ with $|1/\beta^{(g)}_M|$ is sufficiently close to $\gamma_i$ if $M$ is sufficiently large.
Especially, there exists an archimedean place $v_{0,M}$ of $\Q(\beta_M)$ with 
\begin{align*}
&|\beta_M|_{v_{0,M}}\to \infty \ (M\to \infty)\enspace,\\
&|\beta_M|_v\le C_1 \ \text{for} \ v\in {{\mathfrak{M}}}^{\infty}_{\Q(\beta_M)} \ \text{with} \ v\neq v_{0,M}, \ M\in \N \ \text{with} \  (M,p)=1\enspace,
\end{align*}
for some $C_1>0$.
The number $cp\cdot \beta_M$ is algebraic integer for any $M$, especially, we have $${\rm{den}}(\beta_M)\le cp\enspace.$$ 
The family of algebraic numbers $(\beta_M)_{\substack{M\in\N \\(M,p)=1}}$ satisfies the assumptions of Corollary $\ref{Cor4}$. Thus there exists a positive integer  $M_0:=M_0(d,r,m,\boldsymbol{\alpha},C_1,cp)$ which depends on $d,r,m,\boldsymbol{\alpha}$, $C_1$ and $cp$ satisfying, for any $M\ge M_0$, the $rm+1$ numbers~$:$
$$1,{\rm{Li}}_{1}(\alpha_1/\beta_M),\ldots, {\rm{Li}}_{r}(\alpha_1/\beta_M),\ldots, {\rm{Li}}_{1}(\alpha_m/\beta_M),\ldots, {\rm{Li}}_{r}(\alpha_m/\beta_M)\enspace,$$
are linearly independent over $\Q(\beta_M)$. 
\end{example}

\vspace{\baselineskip}
The present paper is organized as follows: firstly, in Section~\ref{construcpade}, we construct the Pad\'e approximants needed for the proof: this is Theorem~\ref{Pade appro Lerch}. In Section~\ref{determinant}, we prove that the determinant is non-vanishing. The main result is Proposition~\ref{decompose Cnum}. First part is devoted to some preparation, which involve valuation estimates and elementary linear algebra. Second part is devoted to the explicit factorization of the determinant of Hermite  type
and makes full use of the formal construction of the Pad\'e approximants. Non-vanishing determinant is then reduced to the non-vanishing of an absolute constant which can be shown by proving some explicit real integral does not vanish (Lemma~\ref{last lemma}).  We then move to analytic estimates in Section~\ref{analy}. The height estimate of the approximants is provided for in Lemma~\ref{majonorme} and the size of the approximation bounded in Lemma~\ref{upper jyouyonew}; thanks to the formal construction of Section~\ref{construcpade}, simple linear algebra estimates are sufficient and we avoid tricky contour integral estimates standard in the theory.  We then proceed to prove a linear independence criterion (Proposition~\ref{critere version II}) and deduce our main result. We conclude by providing a few 
concrete examples of Corollary $\ref{Cor4}$ and Example $\ref{ex1}$ in Section $6$.

\section{Simultaneous Pad\'{e} approximants of Lerch functions}
\label{construcpade}
In this section, we explicitly construct Pad\'{e} approximants of Lerch functions.
First we recall the definition of Pad\'{e} approximants of formal Laurent series. In the following of this section, we denote by $L$ 
an integral domain of characteristic $0$. {{Denote the quotient field of $L$ by $Q(L)$.}}
We define the order function at $z=\infty$, ${\rm{ord}}_{\infty}$;
\begin{align*}
{\rm{ord}}_{\infty}:{{Q(L)((1/z))}}\rightarrow \Z\cup \{\infty\} \, \,\,\,{\text{by}}\,\,\ \ \sum_{k}{\dfrac{a_k}{z^k}}\mapsto \min\{k\in \Z\mid a_k\neq 0\}\enspace.
\end{align*}  
We first recall without proof the following elementary fact~:
\begin{lemma} \label{pade}
Let $r$ be a positive integer, $f_1(z),\ldots,f_r(z)\in 1/zL[[1/z]]$ and $\boldsymbol{n}:=(n_1,\ldots,n_r)\in \N^{r}$.
Put $N:=\sum_{i=1}^rn_i$.
Let $M$ be a positive integer  with $M\ge N$. Then there exists a family of polynomials 
$(P_0(z),P_{1}(z),\ldots,P_r(z))\in L[z]^{r+1}\setminus\{\bold{0}\}$ satisfying the following conditions~$:$
\begin{align*} 
&({\rm{i}}) \ {\rm{deg}}P_{0}(z)\le M\enspace,\\
&({\rm{ii}}) \ {\rm{ord}}_{\infty} (P_{0}(z)f_j(z)-P_j(z))\ge n_j+1 \ \text{for} \ 1\le j \le r\enspace.
\end{align*}
\end{lemma}
\begin{definition}
We use the same notations as in Lemma $\ref{pade}$. 
We say that a family of polynomials $(P_0(z),P_{1}(z),\ldots,P_r(z)) \in L[z]^{r+1}$ satisfying the properties $(i)$ and $(ii)$ {is a} weight $\boldsymbol{n}$ and degree $M$ Pad\'{e} type approximant of $(f_1,\ldots,f_r)$.
For such $(P_0(z),P_{1}(z),\ldots,P_r(z))$, of $(f_1,\ldots,f_r)$, 
we call the family of formal Laurent series $(P_{0}(z)f_j(z)-P_{j}(z))_{1\le j \le r}$ as weight $\boldsymbol{n}$ {and} degree $M$ Pad\'{e} type approximations of $(f_1,\ldots,f_r)$.
\end{definition}

In the following, we fix $x\in L$ and assume $x+k$ are invertible in $L$ for any $k\in \N$.
We now introduce notations for formal primitive, derivation, and evaluation maps.

\begin{notation} \label{notationderiprim}
\begin{itemize}
\item[(i)]  For $\alpha\in L$, We denote by ${\ev}_{\alpha}$ the linear evaluation map $L[t]\longrightarrow L$, $P\longmapsto P(\alpha)$. At a later stage, when several variables are in play and there is a perceived ambiguity on which variable is being specialized, we shall denote the map $\ev_{t\rightarrow \alpha}$.
\item[(ii)] For $P\in L[t]$, we denote by $[P]$ the multiplication by $P$ ($Q\longmapsto PQ$).
\item[(iii)] We also denote by $\inte=\inte_x$ the linear operator $L[t]\longrightarrow L[t]$, defined by $P\longmapsto \frac{1}{t^{1+x}}\kern-1,7pt\int_{0}^{t}{\xi}^xP(\xi)d\xi$ (formal primitive). 
\item[(iv)] We denote by $\deri=\deri_x$ the derivative map $P\longmapsto t^{-x}\tfrac{d}{dt}(t^{x+1}P(t))$, $S_0={\rm{Id}}$, and for $n\geq 1$, by $S_n=S_{n,x}$ the map taking $t^k$ to $\tfrac{(k+x+1)_n}{n!}t^k$ where 
$(k+x+1)_n$ {is the Pochammer symbol }$(k+x+1)_n:=(k+x+1)\ldots(k+x+{n})$ 
(the divided derivative $P\longmapsto \frac{1}{n!}t^{-x}\frac{d^n}{dt^n}(t^{n+x}P)=\tfrac{1}{n!}\left(\tfrac{d}{dt}+x/t\right)^nt^n(P)$), so that $\deri=S_{1}$.
\item[(v)] If $\varphi$ is an $L$-automorphism of an $L$-module $M$ and $k$ an integer, we denote 
$$\varphi^{(k)}:=\begin{cases}
\overbrace{\varphi\circ\cdots\circ\varphi}^{k-\text{times}} & \ \text{if} \ k>0\enspace,\\
{\rm{id}}_M &  \ \text{if} \ k=0\enspace,\\
\overbrace{\varphi^{-1}\circ\cdots\circ\varphi^{-1}}^{-k-\text{times}} & \ \text{if} \ k<0\enspace.
\end{cases}
$$
\end{itemize}
\end{notation}
Finally, for any given { $s\in \Z$}, we introduce, the linear map { $\varphi_s=\varphi_{\alpha,x,s}$}.
\begin{notation} 
{$$\varphi_s=\varphi_{\alpha,x,s}=[\alpha]\circ\ev_{\alpha}\circ \inte^{(s)}_x\enspace.$$ }
\end{notation}
Note that, for any positive integer $s$, any $f\in L[[t]]$, $\varphi_{\alpha,x,s}(f)$ is a formal analogue of 
\begin{equation}
\label{interprim}\dfrac{1}{(s-1)!}\int^{\alpha}_{0}t^{x} f(t){\rm{log}}^{s-1}\dfrac{1}{t}dt\enspace.
\end{equation}
We record for future use and convenience the following elementary facts~:
\begin{facts}
\label{faitselem}
\begin{itemize} 
\item[$(i)$] The map $\inte$ is an isomorphism and its inverse is $\deri$, hence $\varphi_{s}$ is well defined for $\alpha \in L$ and $s\leq -1$.
\item[$(ii)$] For any integers $n_1,n_2\geq 0$, the divided derivatives $S_{n_1}$ and $S_{n_2}$ commute, $S_{n_1}\circ S_{n_2}=S_{n_2}\circ S_{n_1}$.
\item[$(iii)$]  For any integer $s\in\zu$, any $\alpha\in L$, $\varphi_{s}\circ\deri=\varphi_{s-1}$ (recall that $\varphi_s$ depends on $\alpha$).
\item[$(iv)$] By continuity, all the above mentioned maps extend to $L[[t]]$ for the natural valuation.
\item[$(v)$] The kernel of the map $\varphi_{0}$ is the ideal $(t-\alpha)$.
\end{itemize}
\end{facts} 
Using fact~\ref{faitselem}, (iv) above, and assuming the variable $z\in L$, one recovers the classical Lerch function expressed as a natural image by $\varphi_{s}$, for $s\geq 1$~:
\begin{align} \label{Lerch integral rep}
\varphi_{s}\left(\dfrac{1}{z-t}\right)=\Phi_s(x,\alpha/z)\enspace.
\end{align}
\begin{lemma}
\label{key1}
Let $n$ be a positive integer  and $k$ a non-negative integer, one has the following relations valid in ${\rm{End}}_L(L[[t]])$~$:$
\begin{itemize}
\item[$(i)$] $S_{n}=\dfrac{1}{n!}S_{1}\circ(S_{1}+\id)\circ\cdots\circ(S_{1}+(n-1)\id)\enspace \mbox{and} \enspace [t^k]\circ  S_{1}=(S_{1}-k\id)\circ[t^k]\enspace.$
\item[$(ii)$] There exist rational numbers $\{b_{n,m,l}\} \subset \Q$
with $b_{n,m,0}={\frac{(-m)_n}{n!}}$ and, for every $n,m\geq 0$,
$$
[t^m]\circ S_{n}=\sum_{l=0}^nb_{n,m,l}S^{(l)}_{1}\circ[t^m]\enspace.
$$

\end{itemize}
\end{lemma}
\begin{proof} Part  $(i)$ follows  by checking that left hand side and right hand side maps coincide on the basis $\{t^l\}$ of $L[t]$, and extend on $L[[t]]$ since the latter is also a topological basis. Part $(ii)$ is derived from $(i)$. 
Since we have\footnote{If $l=n$, the empty sum in the last line is taken to be  equal to $1$.}
\begin{align*}
[t^m]\circ S_{n}&=\frac{1}{n!}[t^m]\circ S_{1}\circ(S_{1}+\id)\circ\cdots\circ(S_{1}+(n-1)\id)\\
                       &=\frac{1}{n!}(S_{1}-m{\rm{Id}})\circ(S_{1} -(m-1)\id)\circ\cdots\circ(S_{1}+(n-1-m)\id)\circ [t^m]\\
                       &=\frac{1}{n!}\sum_{l=0}^n(-1)^{n-l}\left[\sum_{0\le {j_1<\cdots< j_{n-l}}\le n-1}(m-{j_1})\ldots (m-{j_{n-l}})\right]S^{(l)}_{1}\circ [t^m]\enspace.
\end{align*}
\end{proof}
The following lemma is a key ingredient to construct Pad\'{e} approximants of Lerch functions.
\begin{lemma} \label{key lemma} 
Let $\alpha\in L$, $k\in \N$ and $a(t)\in (t-\alpha)^m$, with $m\geq 1$. Then, for every $s,l\in\zu$ such that $0\leq l-s\leq m-1$, one has $\varphi_{s}\circ\deri^{(l)}(a(t))=0$.

Moreover for $s\ge1$, $j\geq 1$ {integers and $n_1,\ldots,n_j\in \Z_{\ge0}$,
let } $a(t)\in L[t]$, {be} divisible by $(t-\alpha)^{n_1+\cdots+n_j}$. 
Then we have for every integer $0\leq  k\leq \min\{n_i-1,1\leq i\leq j\}$,
$$ 
\varphi_{s}([t^k]\circ S_{n_1}\circ\cdots \circ S_{n_j}(a(t)))=0\enspace.
$$
\end{lemma}
\begin{proof} Using the fact that $\varphi_{s}\circ \deri^{(l)}=\varphi_{0}\circ\deri^{(l-s)}$ (Facts~\ref{faitselem}, $(iii)$), if $a(t)$ belongs to $(t-\alpha)^{l-s+1}$ (and assuming $l-s\geq 0$), by Leibniz rule $\varphi_{s}\circ \deri^{(l)}(a(t))=0$.

Now, taking into account Lemma~\ref{key1},  $(i)$, one notes that $[t^k]\circ S_{n_1}\circ\cdots\circ S_{n_j}=U(\deri) \circ[t^k]$ for some polynomial $U(X)\in\qu[X]$ of degree  $n_1+\cdots+n_j$ {\it and} valuation $\geq j$.  Indeed, for each $k\leq\min\{n_i-1,1\leq i\leq j\}$, $[t^k]\circ S_{n_i}=S_{1}\circ V(S_{1})\circ[t^k]$ with $V(X)\in \qu[X]$ by combining both remarks of the above mentioned lemma. Using the first part, as soon as $-s\leq -1$, $\varphi_{s}\left([t^k]\circ S_{n_1}\circ\cdots\circ S_{n_s}(a(t)\right)=0\enspace.$
\end{proof} 

\vspace{\baselineskip}
Let $m, r$ be positive integers and $\boldsymbol{\alpha}:=(\alpha_1,\ldots,\alpha_m)\in (L\setminus\{0\})^m$ which are pairwise distinct, we also assume the variable $z\in L$.
We study Pad\'{e} approximants {{of the second kind}} of $(\Phi_{s}(x,\alpha_i/z))_{\substack{1\le i \le m \\ 1\le s \le r}}$.
Let $l$ be a non-negative integer with $0\le l \le rm$. For a positive integer  $n$, we define a family of polynomials~:
\begin{align}
&P_l(z)=P_{n,l}(\boldsymbol{\alpha},x|z):={\rm{Eval}}_z\circ S^{(r)}_{n}\left(t^l\prod_{i=1}^m(t-\alpha_i)^{rn}\right)\enspace, \label{Qnl}\\
&P_{l,i,s}(z)=P_{n,l,i,s}(\boldsymbol{\alpha},x|z):=\varphi_{\alpha_i,x,s}\left(\dfrac{P_{n,l}(\boldsymbol{\alpha},x|z)-P_{n,l}(\boldsymbol{\alpha},x|t)}{z-t}\right) \ \text{for} \ 1\le i \le m, 1\le s \le r\enspace.\label{Qnlijsj}
\end{align}
Under the above notations, we obtain the following theorem.
\begin{theorem} \label{Pade appro Lerch} 
For each $0\leq l\leq rm$, the family of polynomials $(P_{l}(z),P_{l,i,s}(z))_{\substack{1\le i \le m \\ 1\le s \le r}}$ forms a weight $(n,\ldots,n)\in \N^{rm}$ and degree $rmn+l$ Pad\'{e} type approximants system of $(\Phi_{s}(x,\alpha_i/z))_{\substack{1\le i \le m \\ 1\le s \le r}}$.
\end{theorem} 
\begin{proof}
By the definition of $P_{l}(z)$, we have
${\rm{deg}}P_{l}(z)=rmn+l$,
so the degree condition is satisfied. We need only to check the valuation condition.
Put $R_{l,i,s}(z):=P_{l}(z)\Phi_{s}(x,\alpha_i/z)-P_{l,i,s}(z)$.
Then, by $(\ref{Lerch integral rep})$, we obtain
$$
%\begin{array}{lcl}
%\displaystyle
R_{l,i,s}(z)
%&
=% &
 \displaystyle P_{l}(z)\varphi_{\alpha_i,x,s}\left(\dfrac{1}{z-t}\right)-P_{l,i,s}(z) \label{residual term decompose}%\\
                                                       %&
                                                       = %& 
                                                       \displaystyle\varphi_{\alpha_i,x,s}\left(\dfrac{P_{l}(t)}{z-t}\right)
                                                       =\sum_{k=0}^{\infty}\frac{\varphi_{\alpha_i,x,s}(t^kP_{l}(t))}{z^{k+1}}\enspace. \nonumber 
                                                       %\end{array}
$$
By {definition, $P_{l}(t)$ is the image by  the morphism $S_{n}\circ\cdots \circ S_{n}$ ($r$ times) of a polynomial} vanishing at $\alpha_i$ at least $rn$ times, so, using Lemma $\ref{key lemma}$ with $j=r$, $n_i=n$ for $i\leq j$,
we get
\begin{align*}
\varphi_{\alpha_i,x,s}(t^kP_{l}(t))=0 \ \text{for} \ 1\le i \le m, 1\le s \le r \ \text{and} \ 0\le k \le n-1\enspace.
\end{align*}
From the expansion above of $R_{l,i,s}(z)$ we obtain
$
{\rm{ord}}_{\infty}R_{l,i,s}(z)\ge n+1 \ \text{for} \ 1\le i \le m, 1\le s \le r\enspace,
$ and Theorem $\ref{Pade appro Lerch}$ follows.
\end{proof}

\section{Proof of Main results}
\label{determinant}
\subsection{Non-vanishing Wronskian of Hermite type}
Let $m,r$ be positive integers, in the sequel, we fix $K=\qu(\alpha_1,\ldots,\alpha_m)$ and  fix $x$ to be some rational number which is {\it not} a strictly negative integer.  
Finally, we set $L=K(z)$.

For a positive integer  $l$ with $0\le l \le rm$, we recall the polynomials $P_l(z),P_{l,i,s}(z)$ defined in $(\ref{Qnl})$ and $(\ref{Qnlijsj})$ respectively.
We define a column vector $\vec{p}_l(z)=\vec{p}_{n,l}(z)\in K[z]^{rm+1}$ by
\begin{align*}
&\vec{p}_{l}(z):={}^t\kern-2pt\Biggl(P_{l}(z),
{P_{l,1,1}(z),\ldots, P_{l,1,r}(z)}, \ldots, {P_{l,m,1}(z),\ldots, P_{l,m,r}(z)}\Biggr)\enspace.
\end{align*}
The aim of this subsection is to prove the following proposition.
\begin{proposition} \label{non zero det}
We use the same notations as above. For any positive integer  $n$, we have
$$
\Delta(z)=\Delta_n(z):=
{\rm{det}}  {\begin{pmatrix}
\vec{p}_{0}(z) \ \cdots \ \vec{p}_{rm}(z)
\end{pmatrix}}
\in K\setminus\{0\}\enspace. 
$$
\end{proposition}
{The proof of this proposition is involved  {and will occupy us for} the next subsections. We shall proceed as follows.

The first reduction is performed in Subsection~\ref{firststep}. We first prove that $\Delta(z)\in K$ ({\it i.~e.} $\Delta(z)$ is a constant independent of $z$) which is summed up by Lemma~\ref{sufficient condition}. This is achieved by a valuation argument, by proving that $\Delta(z)$ is a polynomial in $z$, but of negative valuation with respect to $z$.

Next, we move on to expressing $\Delta=\Delta(z)$ as the image of a given polynomial by our base operators $\varphi$. This is performed in Subsection~\ref{secondstep}. To achieve this step, we introduce ``false'' variables (where the variable $t$ is split in many variables $t_{i,s}$ as there are columns. After showing that the determinant is in fact the determinant of a given minor of the original matrix (all the others canceling out, this is just an extension of the previous valuation argument), we take advantage of the false variables and of the linearity of the operators $\varphi$ to express the determinant as desired: the needed result is described by Lemma~$\ref{another presentation}$ and Lemma~\ref{detpsi}. It can be shown, using Leibniz formula and multilinearity of the determinant, that the few spurious terms which would not be nicely factored as desired, actually cancel out.

We move on to factor $\Delta$ viewing the $\alpha_i$ as variables: this is performed in Section~\ref{thirdstep}, the main result being stated as Proposition~\ref{decompose Cnum}.  After having checked homogeneity and found the trivial monomial factors in $\alpha_i$ (Lemma~\ref{homo}), we need to show that $\Delta$ also factors through the $\alpha_i-\alpha_j$ at the appropriate power. To achieve this factorization, one shows that the derivative of $\Delta$ with respect to any one of the variables $\alpha_i$ vanishes at the other $\alpha_j$ at the appropriate power. Unfortunately, derivation and operators $\varphi$ do not commute properly. We thus measure the defect of commutativity (Lemma~\ref{prelimi}), and proceed to note that a derivative of sufficiently high order {composed with $\varphi$} will look like a {standard}  derivative, because the operator $\varphi$ is essentially an iterated primitive. 

 Since we start with a polynomial vanishing of a high order along $\alpha_i-\alpha_j$, these derivatives tend to produce a lot of vanishing. %({\red{Lemmata ~\ref{racine} and \ref{racinesymetrique}}}).
{However, this argument is not sufficient since derivatives of lower order are not enough to remove all the primitivation built in the operators $\varphi$. %\sout{It should be taken care} 
{This last roadblock is dealt with}  via a symmetry argument that proves the corresponding integrals must vanish.}

At this stage, one is ready to prove factorization: a combinatorics argument is enough to conclude. This is done in two steps, preparation of the combinatorial argument summing up the conditions in {Lemma~\ref{condicombi}}.%

Finally, it will remain to check that the last numerical constant (which depends only on $r,m$ and $n$) does not vanish. This is achieved in Subsection~\ref{laststep}. Again, after reinterpretation of that number in terms of the operators $\varphi$, one observes (induction on $m$) that this number is an integral of a (real) function (explicitly available as a polynomial in the given variables and their logarithm) over a hypercube. 
{After the induction step (Lemma~\ref{reducc}), } the non-vanishing of this integral is proven in Lemma~\ref{last lemma}, which completes the proof of Proposition~\ref{non zero det}.}
\subsubsection{First step} 
\label{firststep}
\begin{lemma} \label{sufficient condition}
For all $n\in \N$, $\Delta\in K$.
\end{lemma}
\begin{proof}
Recall  $R_{l,i,s}(z)=P_{l}(z)\Phi_{s}(x,\alpha_i/z)-P_{l,i,s}(z)$ for $0\le l \le rm$, $1\le i \le m$ and $1\le s \le r$.
For the matrix in the definition of $\Delta(z)$, adding $-\Phi_{s}(x,\alpha_i/z)$ times 
first row to $(i-1)r+s+1$-th row for $1\le i \le m$ and $1\le s \le r$, we obtain 
                     $$ 
                     \Delta(z)=(-1)^{rm}{\rm{det}}
                     {\begin{pmatrix}
                     P_{0}(z) & \dots &P_{rm}(z)\\
                     R_{0,1,1}(z) & \dots & R_{rm,1,1}(z)\\
                     \vdots    & \ddots & \vdots  \\
                     R_{0,1,r}(z) & \dots & R_{rm,1,r}(z)\\
                     \vdots & \ddots & \vdots\\
                     R_{0,m,1}(z) & \dots & R_{rm,m,1}(z)\\
                     \vdots    & \ddots & \vdots  \\
                     R_{0,m,r}(z) & \dots & R_{rm,m,r}(z)\\
                     \end{pmatrix}}\enspace. 
                     $$
We denote the $(s,t)$-th cofactor of the matrix in the right hand side of the above equality by $\Delta_{s,t}(z)$.
Then we have, developing along the first row 
\begin{align} \label{formal power series rep delta}
\Delta(z)=(-1)^{rm}\left(\sum_{l=0}^{rm}P_{l}(z)\Delta_{1,l+1}(z)\right)\enspace.
\end{align} 
Since we obtained in Theorem~\ref{Pade appro Lerch}
\begin{align*}
{\rm{ord}}_{\infty} R_{l,i,s}(z)\ge n+1 \ \text{for} \ 0\le l \le rm, \ 1\le i\le m \ \text{and} \ 1\le s \le r\enspace,
\end{align*}
we have
$$
{\rm{ord}}_{\infty}\Delta_{1,l+1}(z)\ge (n+1)rm\enspace.
$$
Using ${\rm{deg}}P_{l}(z)=rmn+l$ and this lower bound for the valuation of $\Delta_{1,l+1}(z)$, we obtain
\begin{equation}
P_{l}(z)\Delta_{1,l+1}(z)\in 1/zK[[1/z]] \ \text{for} \ 0\le l \le rm-1\ ,
\mbox{ and }
 \label{const+laurent series}
P_{rm}(z)\Delta_{1,rm+1}(z)\in K[[1/z]]\enspace.
\end{equation}
Note that in the above relation, the constant term of $P_{rm}(z)\Delta_{1,rm+1}(z)$ is 
\begin{equation*}
``\text{Coefficient of} \ z^{rm(n+1)} \ \text{of} \ P_{rm}(z)\text{''} \times ``\text{Coefficient of} \ 1/z^{rm(n+1)} \ \text{of} \ \Delta_{1,rm+1}(z)\text{''}\enspace.
\end{equation*}  
Thus, by $(\ref{formal power series rep delta})$, $\Delta(z)$ is a polynomial in $z$ with non-positive valuation with respect to ${\rm{ord}}_{\infty}$, it has to be a constant. 
Moreover, the terms with strictly negative valuation have to cancel out, so 
\begin{align}\label{in K}
\Delta(z)=(-1)^{rm}\kern-2pt\times\kern-2pt\left(\sum_{l=0}^{rm}P_{l}(z)\Delta_{1,l+1}(z)\right)=(-1)^{rm}\times``\text{Constant term of} \ P_{rm}(z)\Delta_{1,rm+1}(z)\text{''} \in K\enspace.
\end{align}
This completes the proof of Lemma $\ref{sufficient condition}$.
\end{proof} 
We can now proceed to the second step (factoring $\Delta$ as an element of $K$).
\subsubsection{Second step} 
\label{secondstep}
We use the same notations as in the proof of Lemma $\ref{sufficient condition}$.
By the equalities $(\ref{const+laurent series})$ and $(\ref{in K})$, we have
{\small{\begin{align} \label{in K2}
\Delta(z)=(-1)^{rm}\times ``\text{Coefficient of} \ z^{rm(n+1)} \ \text{of} \ P_{rm}(z)'' \times ``\text{Coefficient of} \ 1/z^{rm(n+1)} \ \text{of} \ \Delta_{1,rm+1}(z)\text{''}\enspace.
\end{align}}}
Define a column vector $\vec{q}_l={\vec{q}}_{n,l}\in K^{rm}$ by
\begin{align*}
\vec{q}_{l}:={}^t\Biggl(\varphi_{\alpha_1,x,1}(t^nP_{l}(t)),\ldots, \varphi_{\alpha_1,x,r}(t^nP_{l}(t)),\ldots, \varphi_{\alpha_m,x,1}(t^nP_{l}(t)),\ldots, \varphi_{\alpha_m,x,r}(t^nP_{l}(t))\Biggr)\enspace.
\end{align*}
Then by the definition of $\Delta_{1,rm+1}(z)$ and the equalities 
\begin{align*}
R_{l,i,s}(z)=\sum_{k=n}^{\infty}\frac{\varphi_{\alpha_i,x,s}(t^kP_{l}(t))}{z^{k+1}} \kern10pt \text{for} \ 0\le l \le rm \ 1\le i \le m \ \text{and} \ 1\le s \le r\enspace,
\end{align*} 
we have
$$
``\text{Coefficient of} \ 1/z^{rm(n+1)} \ \text{of} \ \Delta_{1,rm+1}(z)\text{''}={\rm{det}}
{\begin{pmatrix}
\vec{q}_{0} \ \cdots \ \vec{q}_{rm-1}
\end{pmatrix}}\enspace.
$$
Thus by $(\ref{in K2})$ and the above equality, we have 
\begin{align} \label{bunkai 0} 
\Delta(z)=(-1)^{rm}\dfrac{1}{(rm(n+1))!}\left(\dfrac{d}{dz}\right)^{rm(n+1)}P_{rm}(z)\cdot
{\rm{det}}{\begin{pmatrix}
\vec{q}_{0} \ \cdots \ \vec{q}_{rm-1}
\end{pmatrix}} \ \text{for all} \ n\in \N\enspace. 
\end{align}
Then, by the definition of $P_{rm}(z)$, we have
$$\dfrac{1}{(rm(n+1))!}\left(\dfrac{d}{dz}\right)^{rm(n+1)}P_{rm}(z)=\left(\dfrac{(1+rm(n+1)+x)_n}{n!}\right)^r \neq 0\enspace.$$
By the equality $(\ref{bunkai 0})$, we study the values
\begin{align*} 
\Theta=\Theta_n:= 
{\rm{det}}  {\begin{pmatrix}
\vec{q}_{0} \ \cdots \ \vec{q}_{rm-1} 
\end{pmatrix}} \ \text{for} \ n\in \N\enspace.
\end{align*}
\vspace{\baselineskip}

Next, we shall need to separate the variables, so, instead of working over the ring $L[t]$, we shall consider as many variables as there are Lerch functions. 
For each variable $t_{i,s}$ and any integer, one has a well defined map
\begin{equation*}
\varphi_{i,s,l}=\varphi_{\alpha_i,t_{i,s},x, l}:L[t_{i,s}]_{1\leq i\leq m,1\leq s\leq r}\longrightarrow L[t_{i',s'}]_{(i',s')\neq (i,s)}; \ \ t^k_{i,s}\mapsto \dfrac{\alpha_i^{k+1}}{(k+x+1)^l}\enspace,
\end{equation*} 
using the \kern-0.5ptdefinition \kern-0.5ptabove \kern-0.5ptwhere $L[\kern-1ptt_{i,s}\kern-1pt]_{\kern-1pt1\leq i\leq m,1\leq s\leq r}$ is seen as the  one variable \kern-0.5ptpolynomial ring over $L\kern-1pt[t_{i',s'}\kern-1pt]_{(i',s')\neq (i,s)}$.

For a positive integer  $n$ and an integer $l$ with $ 0\le l \le rm$, we put 
\begin{align*}
&A_{l}(z)=A_{n,l}(z):=z^l\prod_{i=1}^m(z-\alpha_i)^{rn}\enspace.
\end{align*}
By the definition of $A_{l}(z)$, we have $P_{l}(z)={\rm{Eval}}_z\circ S^{(r)}_{n}\left(A_{l}(t)\right).$

\ 

Define a column vector $\vec{{\rho}}_l=\vec{{\rho}}_{n,l}\in L^{rm}$ by
$$
\vec{{\rho}}_{l}:=
{}^t\kern-.51pt\Biggl(\kern-2pt\varphi_{1,1,1}(t^n_{1,1}A_{l}(t_{1,1})), \ldots, \varphi_{1,1,r}(t^n_{1,r}A_{l}(t_{1,r})), \ldots, \varphi_{m,1,1}(t^n_{m,1}A_{l}(t_{m,1})), \ldots, \varphi_{m,r,r}(t^n_{m,r}A_{l}(t_{m,r}))\kern-2pt\Biggr)\enspace\kern-3pt.
$$
\begin{lemma} \label{another presentation}
We use the notations as above. Then we obtain the following equality~$:$
$$ %\label{another rep}
\Theta_n=(-1)^{r^2mn}
{\rm{det}}  {\begin{pmatrix}
\vec{{\rho}}_{0} \
\cdots \
\vec{{\rho}}_{rm-1}
\end{pmatrix}}\enspace.
$$
\end{lemma}
\begin{proof}
Using Lemma $\ref{key1}$  $(ii)$ {with $m=n$}, there exists a set of rational numbers $\{e_{n,k}\}_{0\le k \le rn}\subset \qu$ with $e_{n,0}=(-1)^{rn}$ and
$$ %\label{decompose t^n S^r_n}
[t^n]\circ S^{(r)}_{n}=\sum_{k=0}^{rn}e_{n,k}S^{(k)}_{1}\circ [t^n]\enspace.
$$
By the above equality, we obtain  
\begin{align}
\varphi_{{\alpha_i,x,s}}(t^nP_{l}(t))&=\sum_{k=0}^{rn}e_{n,k}\varphi_{{\alpha_i,x,s}}\circ S^{(k)}_{1}\circ[t^n](A_{l}(t))\label{calculation}\\
                                         &=\sum_{k=0}^{s-1}e_{n,k}\varphi_{{\alpha_i,s,s-k}}\circ[t^n](A_{l}(t))+\sum_{k=s}^{rn}e_{n,k}\varphi_{{\alpha_i,x,0}}\circ S^{(k-s)}_{1}\circ[t^n](A_{l}(t))\nonumber\\
                                         &=\sum_{k=0}^{s-1}e_{n,k}\varphi_{{\alpha_i,x,s-k}}(t^nA_{l}(t))\enspace, \nonumber %\label{vanish k>s}
\end{align}
for $1\le i \le m$ and $1\le s \le r$. Note, in the last equality as above, we use Lemma $\ref{key lemma}$. 
We define the set of row vectors of length $rm$ with entries $L$ , $\{\vec{v}_{i,s}=\vec{v}_{n,i,s}\}_{\substack{1\le i \le m \\ 1\le s \le r}}$ by
\begin{align*}
\vec{v}_{i,s}:=\left(\sum_{k=0}^{s-1}e_{n,k}\varphi_{\alpha_i,x,s-k}(t^nA_{n,l}(t))\right)_{0\le l \le rm-1}\enspace.
\end{align*}
By $(\ref{calculation})$, we obtain using linear combination of lines
\begin{align*}
\Theta= 
{\rm{det}}{}^t\kern-2pt{\begin{pmatrix}
\vec{v}_{1,1} \ \cdots \ \vec{v}_{1,r} \ \cdots \ \vec{v}_{m,1}  \ \cdots \ \vec{v}_{m,r}
\end{pmatrix}}
=(-1)^{r^2mn}
{\rm{det}}  {\begin{pmatrix} 
\vec{{\rho}}_{0} \ \cdots \ \vec{{\rho}}_{rm-1}
\end{pmatrix}}\enspace.
\end{align*}
This completes the proof of Lemma $\ref{another presentation}$.
\end{proof}
We now define for non-negative integers $u,n$ where $\boldsymbol{t}$ stands for $\{t_{i,s}\}_{1\leq i\leq m,1\leq s\leq r}$
$$
{{P_u(\boldsymbol{t})=P_{u,n, \boldsymbol{\alpha}}(\boldsymbol{t})}}=\prod_{i=1}^{m}\prod_{s=1}^r\left[ t_{i,s}^u\prod_{j=1}^m(t_{i,s}-\alpha_j)^{rn}\right]\prod_{(i_1,s_1)<(i_2,s_2)}(t_{i_2,s_2}-t_{i_1,s_1})\enspace,
$$
where the order $(i_1,s_{1})<(i_2,s_{2})$ means lexicographical order.
Also set (when no confusion is deemed to occur, we omit the subscripts $\boldsymbol{\alpha}=(\alpha_1,\ldots,\alpha_m)$)~:
$$\psi=\psi_{\boldsymbol{\alpha}}:=\bigcirc_{i=1}^{m}\bigcirc_{s=1}^r\varphi_{i,s,s}\enspace.$$
Then, by the definition of 
$
{\rm{det}}  {\begin{pmatrix}
\vec{{\rho}}_{0} \
\cdots \
\vec{{\rho}}_{rm-1}
\end{pmatrix}},
$
we have just {proven~:
\begin{lemma}
\label{detpsi}
$$
{\rm{det}}  {\begin{pmatrix}
\vec{{\rho}}_{0} \
\cdots \
\vec{{\rho}}_{rm-1}
\end{pmatrix}}=\psi(P_{n})\enspace.
$$
\end{lemma}}
\subsubsection{Third step}
\label{thirdstep}
Let $u$ be a non-negative integer and $m$ a positive integer
, we study the value
\begin{align*}
&C_{u,m}=C_{n,u,m}({{\boldsymbol{\alpha}}}):=\psi_{\boldsymbol{\alpha}}(P_{u})=\psi(P_{u})\enspace.
\end{align*} 
The aim of this section is to prove the following proposition.

\begin{proposition} \label{decompose Cnum}
We use above notations. Then there exists a constant $c_{u,m}=c_{n,u,m}\in \qu$ with
$$
C_{u,m}=c_{u,m}\prod_{i=1}^m\alpha^{r(u+1)+r^2n+\binom{r}{2}}_i \prod_{1\le i<j\le m}(\alpha_{j}-\alpha_{i})^{(2n+1)r^2}\enspace.
$$
\end{proposition}
It is also easy to see that since all the variables $t_{i,s}$ have been specialized, $C_{u,m}\in K$ is a polynomial in the $\alpha_i$. The statement is then about a factorization of this polynomial. 
To prove of Proposition $\ref{decompose Cnum}$, we are going to perform the following steps~:
\begin{itemize}
\item[$(i)$] Show that $C_{u,m}$ is homogeneous of degree $m[r(u+1)+r^2n+{r\choose 2}]+{m\choose 2}(2n+1)r^2$.
\item[$(ii)$] Show that $\prod_{i=1}^m\alpha^{r(u+1)+r^2n+\binom{r}{2}}_i$ divides $C_{u,m}$.
\item[$(iii)$] Show that $\prod_{1\le i_1<i_2\le m}(\alpha_{i_2}-\alpha_{i_1})^{(2n+1)r^2}$ divides $C_{u,m}$.
\end{itemize}
We first prove $(i)$ and $(ii)$.
\begin{lemma} \label{homo}
$C_{u,m}$ is homogeneous of degree $m[r(u+1)+r^2n+{r\choose 2}]+{m\choose 2}(2n+1)r^2$ and is divisible by
$\prod_{i=1}^m\alpha^{r(u+1)+r^2n+\binom{r}{2}}_i$.
\end{lemma}
\begin{proof} 
First the polynomial $P_{u}(\boldsymbol{t})$ is a homogeneous polynomial with respect to the variables $\alpha_i,t_{i,s}$ of degree $m[ru+r^2n+{r\choose 2}]+{m\choose 2}(2n+1)r^2$.
By the definition of $\psi$, it is {thus} easy to see that $C_{u,m}=\psi(P_{u}(\boldsymbol{t}))$ is a homogeneous polynomial with respect to the variables $\alpha_i$ of degree $m[r(u+1)+r^2n+{r\choose 2}]+{m\choose 2}(2n+1)r^2$.
Second we show the later assertion. By linear algebra $\varphi_{l}(Q(t))=\alpha \varphi_{1,x,l}(Q(\alpha t))$ ({\it  i.e.}  the variable $t$ specializes in $1$, {\it confer} Lemma~\ref{prelimi} (ii) below) for any integer $l$ and any polynomial $Q\in K[t]$. 

So, by composition, the same holds for $\psi$, and, putting $\boldsymbol{1}=(1,\ldots,1)$, {one gets}
$$C_{u,m}=\prod_{i=1}^m\alpha_i^r\cdot \bigcirc_{i=1}^m\left(\psi_{\boldsymbol{1}}(P_{u}(\alpha_i t_{i,s}))\right)\enspace.$$
We now compute
$$
P_{u}(\alpha_i t_{i,s})=\prod_{i=1}^m\alpha_i^{r^2n+ru+{r\choose 2}}\times  Q({\boldsymbol{t}})\enspace,
$$
where 
\begin{align*}
Q(\boldsymbol{t})=Q_{n,u,m}(\boldsymbol{t})&=\left( \prod_{i=1}^m \prod_{s=1}^r \left[t_{i,s}^u\prod_{j\neq i}(\alpha_it_{i,s}-\alpha_j)^{rn}(t_{i,s}-1)^{rn}\right]\right)\\
& \times \displaystyle{\prod_{1\le i_1<i_2\le m}\prod_{1\le s_1,s_2\le r}}(\alpha_{i_2}t_{i_2,s_2}-\alpha_{i_1}t_{i_1,s_1})\times \displaystyle{\prod_{i=1}^m\prod_{1\le s_1<s_2\le r}}(t_{i,s_2}-t_{i,s_1})\enspace,
\end{align*}
by linearity, 
\begin{equation} \label{equality 1} 
C_{u,m}=\prod_{i=1}^m\alpha_i^{r(u+1)+r^2n+{r\choose 2}}\left(\psi_{\boldsymbol{1}}\left(Q\right)\right)\enspace.
\end{equation}
This concludes the proof of the lemma. 
\end{proof}
{Now we consider $(iii)$.  Since the statement is trivial for $m=1$, we can assume $m\geq 2$. We need to show that $(\alpha_j-\alpha_i)^{(2n+1)r^2}$ divides $C_{u,m}$. Without loss of generality, after renumbering, we can assume that $j=2,i=1$.
 To ease notations, we are going to take advantage of the fact that $m\geq2$, and set $X_s=t_{1,s}, Y_s=t_{2,s}$ and $\alpha_1=\alpha$, $\alpha_2=\beta$.
So our polynomial $P_u$ rewrites as
\begin{align*}
P_{u}(\underline{X},\underline{Y})&=\prod_{s=1}^r\left[(X_sY_s)^u[(X_s-\alpha)(X_s-\beta)(Y_s-\alpha)(Y_s-\beta)]^{rn}\right]\\
& \times \prod_{1\leq i<j\leq r}(X_j-X_i)\prod_{1\leq i<j\leq r}(Y_j-Y_i)\prod_{1\leq i,j\leq r}(Y_j-X_i)
\\ & \times c(t_{i,s})_{i\geq 3}\prod_{k\geq 3}\prod_{s\geq 1}\prod_{1\leq i,j \leq r}(t_{k,s}-X_i)(t_{k,s}-Y_j)\enspace,
\end{align*}
where $c(t_{i,s}):=\prod_{i=3}^{m}\prod_{s=1}^r\left[ t_{i,s}^u\prod_{j=1}^m(t_{i,s}-\alpha_j)^{rn}\right]\prod_{(i_1,s_1)<(i_2,s_2),i_1,i_2\geq 3}(t_{i_2,s_2}-t_{i_1,s_1})$ (the precise value of $c$ does not actually matter as it is treated as a scalar by the operators $\varphi_{\alpha,s,l},\varphi_{\beta,s,l}$).

Similarly, the maps $\varphi_{\alpha_i,t_{i,s},x,l}$ rewrite as $\varphi_{\alpha,s,l},\varphi_{\beta,s,l}$ respectively (for $i=1,2$).
For any $\boldsymbol{l}=(l_1,\ldots,l_r)\in\Z^r$, we set $\psi_{\alpha,\boldsymbol{l}}:=
\bigcirc_{s=1}^r\varphi_{\alpha,s,l_s}$, and $\psi_{\beta,\boldsymbol{l}}:=
\bigcirc_{s=1}^r\varphi_{\beta,s,l_s}$ respectively.
Setting $\boldsymbol{k}=(1,2,\ldots,r)$
one has 
$$\psi=\psi_{\alpha,\boldsymbol{k}}\circ\psi_{\beta,\boldsymbol{k}}\circ\theta\enspace,$$
where $\theta=\bigcirc_{i\geq 3}\bigcirc_{s=1}^r\varphi_{\alpha_i,s,s}$. 

\vspace{\baselineskip}
{{We introduce a specialization morphism for the variable $\alpha$. Set 
$$\Delta=\Delta_{\alpha}: \Q[\alpha,1/\alpha, \beta, \ldots,\alpha_m] \longrightarrow \qu(\alpha); \ \ \Delta(P(\alpha,\beta,\ldots,\alpha_m))=P(\alpha,\ldots,\alpha)\enspace.$$}}
Note that $\theta$ and $\psi$ commute so, it is enough to prove that for $0\le l \le (2n+1)r^2-1$,  one has
{{$\frac{\partial^l}{\partial \alpha^l}\Delta \left(\psi_{\alpha,\boldsymbol{k}}\circ\psi_{\beta,\boldsymbol{k}}(P_u)\right)=0$}}.
We postpone the end of the proof of $(iii)$ and start with a few preliminaries.
We  set 
$$\begin{array}{lcl}\displaystyle
f(\underline{X},\underline{Y})=f(\alpha,\beta,\underline{X},\underline{Y}) & =  & \displaystyle c(t_{i,s})\prod_{k\geq 3}\prod_{s\geq 1}\prod_{1\leq i,j \leq r}(t_{k,s}-X_i)(t_{k,s}-Y_j)\\ & \times & \displaystyle \prod_{s=1}^r(X_sY_s)^u[(X_s-\alpha)(X_s-\beta)(Y_s-\alpha)(Y_s-\beta)]^{rn}\enspace,
\end{array}$$
and 
$$g(\underline{X},\underline{Y})=g(\alpha,\beta,\underline{X},\underline{Y})=\prod_{1\leq i,j\leq r}(X_i-Y_j)\prod_{1\leq i<j\leq r}[(X_j-X_i)(Y_j-Y_i)]\enspace.$$

So that $P_{u}=fg$ (for the rest of the proof, the indexes $u,n$ will not play any role and may be conveniently left off to ease reading).

We now concentrate on a few elementary properties of the maps $\varphi$ which we regroup here and will be useful for the rest~:
\begin{lemma}\label{prelimi}
\begin{itemize}
\item[$(i)$] The morphisms $\varphi_{\alpha,s,l},\varphi_{\beta,s,l}$ pairwise commute for $1\leq s\leq r$, $l\in \Z$.
\item[$(ii)$] For any $a\in L[t]$ and any $s\in \Z$, $\varphi_{s}(a(t))=\alpha\varphi_{1,x,s}(a(\alpha t)).$
\item[$(iii)$] The operator $\frac{\partial}{\partial \alpha}$ commutes with any $\varphi_{\beta,s,l}$ and hence with $\psi_{\beta,\boldsymbol{k}}$ and with $\theta$.
\item[$(iv)$] Define $\boldsymbol{1}_s=(0,\ldots,0,1,0,\ldots,0)$ where $1$ stands at the $s$-th place and zeros are elsewhere.
Then for $\boldsymbol{l}\in\Z^r$ and $f\in K[X_1,\ldots,X_r]$, we have $\frac{\partial}{\partial \alpha}\circ\psi_{\alpha,\boldsymbol{l}}(f)=\psi_{\alpha,\boldsymbol{l}}\circ\left(\frac{\partial}{\partial \alpha}-rx/\alpha\right)(f)+\sum_{s=1}^r\psi_{\alpha,\boldsymbol{l}-\mathbf{1}_s}(f/\alpha)$.
\end{itemize}
\end{lemma}
\begin{proof} $(i)$ follows from the definition since both multiplication by a scalar, specialization of one variable or integration with respect to a given variable all pairwise commute, $(iii)$ follows from commutation of integrals with respect to a parameter with differentiation with respect to that parameter.

We turn to $(ii)$. Take $k\geq 0$, then  $\varphi_{s}(t^k)=[\alpha]\circ\ev_{\alpha}\circ\inte^{(s)}(t^k)=[\alpha]\circ\ev_\alpha\left(\frac{t^{k}}{(k+x+1)^s}\right)=\frac{\alpha^{k+1}}{(k+x+1)^s}$. 
Similarly, $\alpha \varphi_{1,x,s}((\alpha t)^k)=\frac{\alpha^{k+1}}{(k+x+1)^s}$. Since both coincide, for the basis $\{t^k\}_{k\in \N}$, they coincide on $L[t]$ by linearity.

Finally, we prove $(iv)$. 
Let $f(X_1,\ldots, X_r)=\sum_{\underline{i}}a_{\underline{i}}(\alpha)\prod_{j=1}^{r}X_{j}^{i_j}\in K[X_1,\ldots, X_r]$.
Then  by definition, 
$\psi_{\alpha,\boldsymbol{l}}(f)=\sum_{\underline{i}}a_{\underline{i}}(\alpha)\frac{\alpha^{\vert \underline{i}\vert +r}}{\prod_{j=1}^r(i_j+x+1)^{l_j}}$.
Here, $\vert \underline{i}\vert=\sum_{j=1}^ri_j$. 
Then
$$\frac{\partial}{\partial \alpha}(\psi_{\alpha,\boldsymbol{l}}(f))=\sum_{\underline{i}}\frac{\partial}{\partial \alpha}(a_{\underline{i}}(\alpha))\frac{\alpha^{\vert \underline{i}\vert +r}}{\prod_{j=1}^r(i_j+x+1)^{l_j}}+ \sum_{\underline{i}}a_{\underline{i}}(\alpha)(\vert \underline{i}\vert+r)\frac{\alpha^{\vert\underline{i}\vert +r-1}}{\prod_{j=1}^r(i_j+x+1)^{l_j}}\enspace.$$

First term in the sum is easily seen to be equal to $\psi_{\alpha,\boldsymbol{l}}(\frac{d}{d \alpha}(f))$.
So we look at
$$\sum_{\underline{i}}a_{\underline{i}}(\alpha)(\vert \underline{i}\vert+r)\frac{\alpha^{\vert\underline{i}\vert +r-1}}{\prod_{j=1}^r(i_j+x+1)^{l_j}}
=\sum_{\underline{i}}\frac{a_{\underline{i}}(\alpha)}{\alpha}(\vert \underline{i}\vert+r)\frac{\alpha^{\vert\underline{i}\vert +r}}{\prod_{j=1}^r(i_j+x+1)^{l_j}}\enspace.$$

So the claim $(iv)$ reduces to the statement
$$\left(\sum_{s=1}^r\psi_{\alpha,\boldsymbol{l}-\mathbf{1}_s}-rx\psi_{\alpha,\boldsymbol{l}}\right)(X_1^{i_1}\ldots X_r^{i_r})=
(\vert \underline{i}\vert+r)\frac{\alpha^{\vert\underline{i}\vert +r}}{\prod_{j=1}^r(i_j+x+1)^{l_j}}\enspace.$$

But left hand side is 
\begin{align*}
\left(\sum_{s=1}^r\psi_{\alpha,\boldsymbol{l}-\boldsymbol{1}_s}-rx\psi_{\alpha,\boldsymbol{l}}\right)(X_1^{i_1}\ldots X_r^{i_r})&=
\sum_{s=1}^r\frac{(i_s+x+1)\alpha^{\vert \underline i\vert +r}}{\prod_{j=1}^r(i_j+x+1)^{l_j}}-rx\frac{\alpha^{\vert\underline{i}\vert +r}}{\prod_{j=1}^r(i_j+x+1)^{l_j}}\\
&=(\vert \underline{i}\vert+r)\frac{\alpha^{\vert\underline{i}\vert +r}}{\prod_{j=1}^r(i_j+x+1)^{l_j}}\enspace.
\end{align*}
\end{proof}
This property will be useful to conclude the proof of the proposition. Indeed, to prove that $C_{u,m}$ is divisible by a high power of $(\alpha-\beta)$, it is enough to show that the 
$\frac{\partial^l}{\partial\alpha^l}(C_{u,m})$ vanishes for many consecutive values of $l$ and since differentiation by $\alpha$ and evaluation at $\alpha$ do not commute properly, it is useful to control the defect of commutativity which is precisely what property (iv) provides for.

\bigskip

Now, let us compute what comes out by iteration of property (iv) of Lemma~\ref{prelimi}. 
Let $l,k\in \Z_{\ge0}$ with $l\ge k$. 
We define a set of differential operators 
$$\mathcal{X}_{l,k}:=\left\{V=\partial_1\circ \ldots \circ \partial_l\mid \partial_i\in \left\{1/\alpha, \tfrac{\partial}{\partial \alpha}-rx/\alpha\right\}, \ \left|\{1\le i \le l , \partial_i=1/\alpha\}\right|=k\right\}\enspace.$$
One gets that 
$$\frac{\partial^l}{\partial\alpha^l}(\psi_{\alpha,\boldsymbol{k}}(P))=\sum_{\substack{\boldsymbol{I}\in \Z^r_{\ge0} \\ \vert \boldsymbol{I}\vert\leq l}}\sum_{V\in \mathcal{X}_{l,\vert \bold{I} \vert}} \psi_{\alpha,\boldsymbol{k}-\boldsymbol{I}}(V(P))\enspace.$$

By the Leibniz formula, for $V\in \mathcal{X}_{l,\vert \bold{I}\vert }$, $V(P)$ is a linear combination of the derivatives 
$\frac{\partial^j}{\partial{\alpha^j}}(P)$ for $0\leq j\leq l-\vert \mathbf{I}\vert$.
Hence, 
setting $P=fg$, it is a linear combination of $g\frac{\partial^j}{\partial{\alpha^j}}(f)$, for $0\leq j\leq l-\vert\mathbf{I}\vert$.
We now pursue the study of the above equation, in the special case where $P=fg$ and perform combinatorics argument~:

\begin{lemma} \label{condicombi}
Let $0\leq l$ be an integer and $\boldsymbol{I}=(a_1,\ldots,a_r)\in \N^r$ such that $\vert \boldsymbol{I}\vert \leq l$.
Assume further either of these two to be true  

\begin{itemize}\item[$(i)$] The  $2r$ dimensional vector $(\boldsymbol{k},\boldsymbol{k}-\boldsymbol{I})$ has two coordinates in common.
\item[$(ii)$] There exists an index $1\leq s\leq r$ such that $0\leq a_s-s< 2rn-l+\vert\boldsymbol{I}\vert$.
\end{itemize}
Then, $\Delta_{\alpha}\circ \psi_{\beta,\boldsymbol{k}}\circ\psi_{\alpha,\boldsymbol{k}-\boldsymbol{I}}(g\frac{\partial^jf}{\partial\alpha^j})=0$ for all $0\leq j\leq l-\vert\mathbf{I}\vert$.

{{Moreover, the smallest integer $l$ for which there exists $\boldsymbol{I}$ with $\vert\boldsymbol{I}\vert\leq l$ with none of the conditions of $(i)$ and $(ii)$ are satisfied is $(2n+1)r^2$.}} 
\end{lemma}
\begin{proof} 
Assume first condition is satisfied and there exist $1\le i,j\le r$ with $k_i=k_j-a_j$.                    
Let $\tau$ be the transposition $\tau(X_i)=Y_j$, $\tau(Y_j)=X_i$ leaving all the other variables invariant and let $\tau$ act on the same way on $(\boldsymbol{k},\boldsymbol{k}-\boldsymbol{I})$ {\it i.e.} $\tau(k_i)=k_j-a_j$, $\tau(k_j-a_j)=k_i$ and all other coordinates invariant. 
Then $\tau$ acts on $L[\underline{X},\underline{Y}]$ by permutation of the variables. Then, by antisymmetry of $g$, we have $\tau(g\frac{\partial^jf}{\partial\alpha^j})=-g\frac{\partial^jf}{\partial\alpha^j}$. 
We compute
\begin{align*}
\Delta_{\alpha}\circ\psi_{\alpha,\boldsymbol{k}}\circ\psi_{\beta,\boldsymbol{k}-\boldsymbol{I}}(g\tfrac{\partial^jf}{\partial\alpha^j})&=\Delta_{\alpha}
\circ \psi_{\alpha,\tau(\boldsymbol{k})}\circ\psi_{\beta,\tau(\boldsymbol{k}-\boldsymbol{I})}(\tau(g\tfrac{\partial^jf}{\partial\alpha^j}))\\
&=\Delta_{\alpha}\circ\psi_{\alpha,\boldsymbol{k}}\circ\psi_{\beta,\boldsymbol{k}-\boldsymbol{I}}(\tau(g\tfrac{\partial^jf}{\partial\alpha^j}))
=-\Delta_{\alpha}\circ\psi_{\alpha,\boldsymbol{k}}\circ \psi_{\beta,\boldsymbol{k}-\boldsymbol{I}}(g\tfrac{\partial^jf}{\partial\alpha^j})\enspace.
\end{align*}

Similarly, if $\boldsymbol{k}$ {or $\boldsymbol{k}-\boldsymbol{I}$} has at least two equal coordinates, we have 
{{$\Delta_{\alpha}\circ \psi_{\beta,\boldsymbol{k}}\circ\psi_{\alpha,\boldsymbol{k}-\boldsymbol{I}}(g\frac{\partial^jf}{\partial\alpha^j})=0.$}}

\bigskip

If the second condition is satisfied, then $\Delta_{\alpha}\circ\varphi_{\alpha,s,s-a_s}(V fg)$ itself vanishes 
for $V\in \mathcal{X}_{l,\vert \boldsymbol{I} \vert}$
since $\left[(X_s-\alpha)(X_s-\beta)\right]^{rn}\mid f$ (so $\frac{\partial^j}{\partial\alpha^j}(fg)$ vanishes at $\alpha=\beta$ at order at least 
$2nr-j\geq 2nr-l+\vert\boldsymbol{I}\vert>a_s-s$).

\bigskip 

Assume conditions $(i)$ and  $(ii)$ are false, then the set $\{a_s-s\}$, {after reordering its elements by increasing order is at least
$\{2rn-l+\vert\boldsymbol{I}\vert;\ldots;2rn-l+\vert\boldsymbol{I}\vert+r-1\}$ in the lexicographic order} and $\sum_{s=1}^r(a_s-s)\geq  r(2rn-l+\vert\boldsymbol{I}\vert)+r(r-1)/2$, that is 
$\vert \boldsymbol{I}\vert+r(l-\vert\boldsymbol{I}\vert)\geq2 r^2n+r^2$. 
Since $l-\vert\boldsymbol{I}\vert\geq 0$, the lemma follows.
\end{proof}

{\noindent \bf End of the proof of $(iii)$.}
By symmetry in the $\alpha_i$,
we need only to show that $(\alpha-\beta)^{(2n+1)r^2}\mid C_{u,m}$. But , since $\theta$ and $\frac{d}{d\alpha}$ commute, the {{later assertion of}} Lemma~\ref{condicombi} ensures that 
$$\left.\frac{\partial^l}{\partial\alpha^l}(C_{u,m})\right|_{\alpha=\beta}=0 \hspace{15pt} \text{for all} \ 0\leq l\leq (2n+1)r^2-1\enspace. \qed$$
}

\subsubsection{Last step}
\label{laststep}
We shall reduce by induction the non-vanishing of $c_{n,u,m}$ to the non-vanishing of $c_{n,u,0}$ (which is obviously equal to 1) and the non-vanishing of a real integral and first prove,
\begin{lemma}\label{reducc} 
Set $\displaystyle A(\boldsymbol{t})=\kern-3pt{\prod_{s=1}^{r}}\kern-3pt \left[t^{u}_{m,s}\kern-2pt \cdot\kern-2pt (t_{m,s}-1)^{rn}\right] \times\kern-10pt\prod_{1\le s<s'\le r}\kern-10pt(t_{m,s}-t_{m,s'})$ and $\theta_m:=\bigcirc_{1\leq s\leq r}\varphi_{1, t_{m,s},x,s} $. Then, 
$$c_{n,u,m}=(-1)^{r^2n(m-1)}c_{n,u+r(n+1),m-1}\cdot\theta_m\left(A(\boldsymbol{t}))\right)\enspace.$$
\end{lemma}
\begin{proof} 
Set $\Theta=\bigcirc_{i=1}^{m-1}\bigcirc_{1\leq i\leq m-1}\bigcirc_{1\leq s\leq r}\varphi_{1,t_{i,s},x,s}$ so that $\psi_{\bold{1}}=\Theta\circ\theta_m$ and
recall that  by $(\ref{equality 1})$,
$$ \label{deco Cnum2} 
D_{n,u,m}:=\dfrac{C_{n,u,m}}{\prod_{i=1}^m\alpha^{r(u+1)+r^2n+\binom{r}{2}}_i}=c_{n,u,m}\prod_{1\le i<j\le m}(\alpha_{j}-\alpha_{i})^{(2n+1)r^2}=\psi_{\boldsymbol{1}}\left(Q_{n,u,m}\right)\enspace.
$$
We are going to evaluate $D_{n,u,m}$ at  $\alpha_m=0$  and thus separate the variables in $Q_{n,u,m}$ first. By definition, one has
\begin{equation*}
Q_{n,u,m}(\boldsymbol{t})=Q_{n,u,m-1}(\boldsymbol{t})\cdot A(\boldsymbol{t})B(\boldsymbol{t})\enspace,
\end{equation*} 
where 
$$B(\boldsymbol{t})=\prod_{s=1}^{r}\prod_{j\neq m}(\alpha_mt_{m,s}-\alpha_j)^{rn}\times
\prod_{i=1}^{m-1} \prod_{s=1}^r (\alpha_it_{i,s}-\alpha_m)^{rn}
\times  {\prod_{1\le i< m}\prod_{1\le s,s'\le r}}\kern-8pt(\alpha_{m}t_{m,s'}-\alpha_{i}t_{i,s})\enspace.
$$
Note that $Q_{n,u,m-1}, A$ do not depend on $\alpha_m$, and $\psi_{\boldsymbol{1}}$ treats $\alpha_m$ as a scalar.
Hence,
\begin{align} \label{compare 1}
\displaystyle  {D_{n,u,m}}_{\vert \alpha_m=0} &= \displaystyle c_{n,u,m}\prod_{i=1}^{m-1}(-\alpha_i)^{(2n+1)r^2}\prod_{1\le i<j\le m-1}(\alpha_{j}-\alpha_{i})^{(2n+1)r^2}\enspace\\
&=\psi_{\boldsymbol{1}}\left(Q_{n,u,m-1}(\boldsymbol{t})A(\boldsymbol{t})B(\boldsymbol{t})_{\vert \alpha_m=0}\right)\enspace. \nonumber
\end{align}
But 
$$B(\boldsymbol{t})_{\vert\alpha_m=0}\kern-1pt=\kern-3pt\prod_{j=1}^{m-1}\kern-2pt(-\alpha_j)^{r^2n}\kern-3pt\prod_{i=1}^{m-1}\kern-2pt\prod_{s=1}^r(\alpha_it_{i,s})^{rn\kern-4pt}\prod_{i=1}^{m-1}\kern-3pt\prod_{s=1}^{r}(\kern-1pt-\alpha_it_{i,s})^r\kern-2pt=\kern-2pt(-1)^{r^2(m-1)(n+1)}\kern-3pt\prod_{i=1}^{m-1}\alpha_i^{(2n+1)r^2}\kern-2pt\prod_{i=1}^{m-1}\kern-3pt\prod_{s=1}^rt_{i,s}^{r(n+1)}\enspace\kern-3pt.$$
We now note  that $\theta_{m}$ treats the variables $t_{i,s}, 1\leq i\leq m-1$ as scalars and $\Theta$ treats variables $t_{m,s}$ as scalars and remark $$Q_{n,u,m-1}(\boldsymbol{t})B(\boldsymbol{t})_{\vert \alpha_m=0}=(-1)^{r^2(m-1)(n+1)}\prod_{i=1}^{m-1}\alpha_i^{(2n+1)r^2}Q_{n,u+r(n+1),m-1}{{(\boldsymbol{t})}}\enspace.$$
Thus
$$\psi_{\boldsymbol{1}}\left(Q_{n,u,m-1}(\boldsymbol{t})A(\boldsymbol{t})B(\boldsymbol{t})_{\vert \alpha_m=0}\right)
=(-1)^{r^2(m-1)(n+1)}\prod_{i=1}^{m-1}\alpha_i^{(2n+1)r^2}\Theta(Q_{n,u+r(n+1),m-1}(\boldsymbol{t}))\theta_m(A(\boldsymbol{t}))\enspace.$$
Using the relation~(\ref{compare 1}), taking into account $D_{n,u+r(n+1),m-1}=\psi_{\boldsymbol{1}}(Q_{n,u+r(n+1),m-1}(\boldsymbol{t}))$ and simplifying,
$$c_{n,u,m}=
(-1)^{r^2n(m-1)}c_{n,u+r(n+1),m-1}\cdot \theta_m(A(\boldsymbol{t}))\enspace.$$
\end{proof}

\begin{lemma} \label{last lemma}
Let $u,n,r$ be positive integers, put $D=[0,1]^r$.
Then we have
$$
\theta_m(A(\boldsymbol{t}))=
\prod_{s^{\prime}=1}^{r}\dfrac{1}{(s^{\prime}-1)!}\int_{D}{\prod_{s=1}^{r}} \left[t^{u+x}_{s} (t_{s}-1)^{rn}{\rm{log}}^{s-1}\dfrac{1}{t_{s}}\right] \cdot \prod_{1\le s_{1}<s_{2}\le r}(t_{s_{2}}-t_{s_{1}})\prod_{s=1}^{r}dt_{s}\neq 0\enspace. 
$$
\end{lemma}
\begin{proof} 
First equality follows from relation~(\ref{interprim}).
We denote the $r$-th symmetric group by $\mathfrak{S}_r$ and put
\begin{align*}
&D_{\sigma}:=\{(t_1,\ldots ,t_r)\in D\mid t_{\sigma(1)}\le \ldots \le t_{\sigma(r)}\} \ \text{for} \ \sigma\in \mathfrak{S}_r\enspace.
\end{align*}
Then we obtain, denoting $g(\boldsymbol{t})$ the function under the integral and $d\mu$ of Lebesgue measure on $D$ 
and using  the change of variables
$\sigma^{-1}:\kern5pt
D_{{\rm{id}}} \longrightarrow D_{\sigma},  \ (t_1,\ldots,t_r) \mapsto (t_{\sigma^{-1}(1)}, \ldots, t_{\sigma^{-1}(r)}).
$
$$
\int_Dgd\mu=
\sum_{\sigma\in\mathfrak{S}_r}\int_{D_{\sigma}}gd\mu=
{\int}_{D_{{\rm{id}}}}
\sum_{\sigma\in\mathfrak{S}_r}g\circ\sigma^{-1}d\mu\enspace;$$
now, note that $\prod_{1\leq s_1<s_2\leq r}(t_{\sigma^{-1}(s_2)}-t_{\sigma^{-1}(s_1)})=\sign(\sigma)\prod_{1\leq s_1<s_2\leq r}(t_{s_2}-t_{s_1})$ and remark (Vandermonde)
\begin{align*}
\sum_{\sigma\in \mathfrak{S}_r}{\rm{sign}}(\sigma)\prod_{s=1}^r{\rm{log}}^{s-1}\left(\dfrac{1}{t_{\sigma(s)}}\right)&=(-1)^{r(r-1)/2}\sum_{\sigma\in \mathfrak{S}_r}{\rm{sign}}(\sigma)\prod_{s=1}^r{\rm{log}}^{s-1}(t_{\sigma(s)})\\
&=(-1)^{r(r-1)/2}\prod_{1\le s_1<s_2\le r}({\rm{log}}(t_{s_2})-{\rm{log}}(t_{s_1}))\enspace,
\end{align*}
so that we can factor
$$
\sum_{\sigma\in\mathfrak{S}_r}g\circ\sigma^{-1}=(-1)^{r(r-1)/2}\prod_{s=1}^rt^{u+x}_s(t_s-1)^{rn}\prod_{1\le s_1<s_2\le r}(t_{s_2}-t_{s_1})\prod_{1\leq s_1<s_2}(\log(t_{s_2})-\log(t_{s_1}))\enspace.
$$
Since the sign of the left hand side is constant equal to $(-1)^{r^2n+r(r-1)/2}$, the integral is non-zero.
\end{proof}

\section{Estimates}
\label{analy}

{We need to estimate $\vert P_l(\beta)\vert_{{{v}}}, \vert P_{l,i,s}(\beta)\vert_{{{v}}}$ as well as the remainder of the Pad\'e approximation.  Instead of taking the standard route of using the explicit value of these polynomials, we take advantage of the fact that the polynomials involved are all images of the elementary polynomial $t^l\prod_{i=1}^m(t-\alpha_i)^{rn}$ by linear operators, and use standard operator norms.}

In this section, we use the following notations. Let $x\in \Q\cap [0,1)$ and 
$K$ be a number field. Let $v$ be a place of $K$ and $K_v$ be the  completion of $K$ at $v$, $\vert\cdot\vert_v$ the absolute value corresponding to $v$.
Let $I$ be a non-empty finite set of indices, $A=K[\alpha_i]_{i\in I}[z,t]$ be a polynomial ring in indeterminate $\alpha_i,z,t$. 

\bigskip

We set $\Vert P\Vert_v=\max\{\vert a\vert_v\}$ where $a$ runs in the coefficients of $P$. Thus $A$ is endowed with a structure of normed vector space. If $\phi$ is {a continuous} endomorphism of $A$, we denote by $\Vert \phi\Vert_{v}$ the endomorphism norm defined in a standard way $\Vert \phi\Vert_v=\inf\{M\in\ru, \forall\kern3pt x\in A, \kern3pt\Vert\phi(x)\Vert_v\leq M\Vert x\Vert_v\}=\sup\left\{\frac{\Vert\phi(x)\Vert_v}{\Vert x\Vert_v}, 0\neq x\in A\right\}$.
This norm is well defined provided $\phi$ is continuous. Unfortunately, we will have to deal  also with non-continuous morphisms {for which of course an endomorphism norm cannot be defined.}

In such a situation, we {force continuity by restricting} the source space to some appropriate  {finite dimensional} sub-vector space $E$ of $A$ and talk of $\Vert\phi\Vert_v$ with $\phi$ seen  as $\left.\phi\right|_{E}: E\longrightarrow A$ on which $\phi$ is {of course} continuous. In case of perceived ambiguity, it will be denoted by $\Vert\phi\Vert_{E,v}$.
The degree of an element of $A$ is as usual the total degree. We first record some elementary facts.

{\begin{lemma} $($cf. \rm{\cite[Lemma $2.2$]{B}}$)$ \label{denominator 1} 
Let $x$ be a non-zero rational number and $n$ a positive integer. 
Put
$$
\mu_n(x):={\rm{den}}(x)^n\prod_{\substack{q:\rm{prime} \\ q|{\rm{den}}(x)}}q^{\lceil n/(q-1) \rceil}\enspace.
$$
Then we have
\begin{itemize}
\item[$(i)$] Let $x=a/b\in \qu$, assume $a,b\in\zu$ are coprime. Let $k\geq 0$ and integer,  and $v$ be a finite place of $K$. Then
$$\frac{\vert(k+1+x)_n\vert_v}{\vert n!\vert_v}\leq \left\{\begin{array}{ll} 1 & \mbox{if $v\nmid b$}\\ \frac{1}{\vert b^n\cdot n!\vert_v} & \mbox{if $v\mid b$}\enspace.
\end{array}\right.$$
\item[$(ii)$]
$
\mu_n(x)\dfrac{(x)_n}{n!}\in \Z.
$

\item[$(iii)$] ${\displaystyle{\lim_{n\rightarrow\infty}}} \mu_n(x)^{1/n}=\mu(x)$,   ${\displaystyle{\lim_{n\rightarrow\infty}}} \vert\mu_n(x)\vert_v^{1/n}=\vert\mu(x)\vert_v$.
\item[$(iv)$] For any integers $k,l$, ${\displaystyle{\lim_{n\rightarrow\infty}}} d_{kn+l}^{1/n}=\exp(k)$.
\end{itemize}
\end{lemma}
 
\begin{proof} If $v\mid b$, $\vert i+a/b\vert_v=\vert a/b\vert_v$ for all $1\leq  i\leq n$ and the lemma follows by the ultrametric inequality. If $v\nmid b$, the result (i) follows from \cite{B}, Lemma $2.2$. {Point}
 (ii) follows from (i) and the well-known equality $v_p(n!)=\tfrac{n}{p-1}+o(n) \ \ (n\to \infty),$ for $p$ a prime number and $v_p$ the $p$-adic normalized {valuation}. 
 (iii) is immediate and (iv) follows from {{the prime number theorem}}.
\end{proof}}

\begin{lemma} \label{fundamental}
\label{norme} We have the following norm estimates $($we do hope the similarity of notations is not cause of confusion$) :$
\begin{itemize}
\item[$(i)$] Let $\alpha\in \{\alpha_i,i\in I, z,t\}$ one of our given variables. Then the multiplication $[\alpha]: P\longmapsto \alpha\cdot P$ is of norm $\Vert [\alpha]\Vert_v=1$. This map conserves {homogeneity}, partial degrees,  except in the variable $\alpha$ where it increases it by 1 $(${and total degree increases by one}$)$.
\item[$(ii)$] Assume $E_N$ is the subspace of polynomials of degree $\leq N$ in each of the variables, and $I$ has cardinality $m$, then the evaluation map $\ev_{\alpha_j}:P(\boldsymbol{\alpha},z,t)\longmapsto 
P(\boldsymbol{\alpha},z,\alpha_j)$ satisfies $\Vert \ev_{\alpha_j}\Vert_{E_N,v}\leq (N+1)^{{{d_v\varepsilon_v(m+2)/d}}}$, where $\varepsilon_v=1$ if $v$ is archimedean and $0$ otherwise. In the particular case where $P$ is a constant in the specialized variable $\alpha_j$, $\Vert\ev_{\alpha_j}\Vert_v\leq 1$. This maps is degree decreasing, {and if $P$ is homogeneous, it preserves homogeneity and degree\footnote{We use the convention that the zero polynomial is homogeneous of any degree.}}.
\item[$(iii)$] Let $E_0=K[\alpha_i,z]$ the sub-vector space of $A$ consisting of constants in $t$, and $\Theta: E_0\longrightarrow A$ the map defined by $\Theta(P)=\frac{P(\alpha_i,z)-P(\alpha_i,t)}{z-t}$. 
Then, $\Vert \Theta\Vert_{E_0,v}=1$. This map is degree decreasing.  {If $P$ is homogeneous, it preserves homogeneity and degree is decreased by 1}.
\item[$(iv)$] Let $I$ be of cardinality $m$,  $E_N$ be the sub-vector space of $B=K[y_1,\ldots,y_m, z,t]$ consisting of polynomials of degree at most $N$ in the variables $y_i$ and $\Gamma: E_N\longrightarrow A$ the morphism defined by $\Gamma(P(y_1,\ldots,y_m,z,t))=P(t-\alpha_1,\ldots,t-\alpha_m,z,t)$. Then, $\Vert\Gamma\Vert_{E_N,v}\leq \left(2^{Nm}(N+1)^{m}\right)^{\varepsilon_v{{d_v/d}}}$. This map conserves each partial degree except possibly in the variable $t$ where the image satisfies $\deg_t(\Theta(P))\leq \sum_{i=1}^m\deg_{\alpha_i}(P)+\deg_{t}(P)$, the total degree in however left unchanged. {If $P$ is homogeneous, $\Gamma(P)$ is homogeneous.}
\item[$(v)$] Let $E_N$ be the subspace of $A$ consisting of polynomials of degree at most $N$ in $t$. Then for all $n\geq 1$ and {$x\in \Q\cap [0,1)$}, $S_n=S_{n,x}:E_N\longrightarrow A$ as defined in Notation~$\ref{notationderiprim}$ satisfies 
{{$$\Vert S_{n}\Vert_{E_N,v}\leq 
\begin{cases}
\left\vert \tfrac{(N+x+1)_n}{n!} \right\vert_v & \ \text{if} \ v: {\text{is archimedian or}  \ v\text{ is non-archimedian}} \ \& \  |x|_v>1\\ 
1 & \ \text{otherwise} \enspace. 
\end{cases}
$$
}}
It acts diagonally on $A$ in the sense that each element of the canonical basis consisting of all monomials is an eigenvector for $S_{n}$. This map conserves degrees {and homogeneity}.
\item[$(vi)$]  Let $E_N$ be the subspace of $A$ consisting of polynomials of degree at most $N$ in $t$. Let $x=a/b\in \Q\cap [0,1)$, with $(a,b)=1$.
Then $\inte=\inte_x\kern-1pt:\kern-1ptE_N\longrightarrow \kern-1ptA$ as defined in Notation~$\ref{notationderiprim}$ satis\-fies 
$\Vert \inte\Vert_{E_N,v}\kern-1pt\leq \max\{\left\vert 1/(k+x+1)\right\vert_v,0\leq k\leq N\}\kern-2pt\leq \kern-1pt\vert d_{b(N+1)+a} \vert_v^{-{ (1-\varepsilon_v)}}\kern-1pt$ where $d_N={\rm{l.c.m.}}(1,\ldots,N)$.
It also acts diagonally on $A$. This map conserves degrees {and homogeneity}.
\end{itemize}
\end{lemma}
\begin{proof} Fact (i) follows from the fact that $[\alpha]$ simply shifts the monomials, and thus is an isometry  onto its image for the sup norm. Fact (ii) follows from the fact that specialization of one variable into another at worse induces collapsing of coefficients, in the ultametric case, the resulting norm is smaller, in the archimedean case, the norm is at worst multiplied by the number of coefficients of the original polynomial collapsing in the image, which is no larger than the dimension of the source space $($and if the original polynomial is a constant in the specialized variable, there is of course no collapse$)$. Fact (iii) follows from the identity $x^k-y^k=(x-y)\left(\sum_{i=0}^{{k-1}}x^iy^{{k-1}-i}\right)$, by looking at the degrees, one sees that the coefficients of $\Theta(P)$ arise from those of $P$ with repetition and no possible collapsing. Fact (iv)  follows from the binomial expansion of $(z-\alpha_i)^k=\sum_{j=0}^k{k\choose j}z^j\alpha_i^{k-j}$ and $\left\vert {k\choose j}\right\vert_v\leq 2^{\varepsilon _v{{d_v/d}}k}$. Since $k\leq N$ and since $m$ variables are changed by $\Gamma$, one gets the desired inequality, taking into account possible collapsing,  the number of which is at most the dimension of  $K_N[y_1,\ldots,y_m]$ (polynomials in $y_i$ of degree at most $N$). Finally (v), and (vi) follow from the definition of $S_{n}$ and $\inte$.
\end{proof}
{
\begin{remark}
The fact that the operators $S_n$ and $\inte$ act diagonally  on the canonical basis is not actually used in the present paper except perhaps to make the norm estimate easier (just compute the largest eigen value). It is nevertheless recorded here for further reference\footnote{See remark~\ref{diago} for further comments on this point.}.
\end{remark}}

\begin{lemma} \label{majonorme} 
Let $x=a/b\in \Q\cap [0,1)$ as above. Then, one has~$:$
\begin{itemize}
\item[$(i)$] The polynomial $P_l(z)=P_{n,l}(\boldsymbol{\alpha},x\vert z)$ satisfies 
$$\Vert P_{l}(z)\Vert_v\leq 
\begin{cases}
\exp\left(\dfrac{nd_v}{d}\left(\left[{\const}\right]+o(1)\right)\rule{0mm}{4mm}\right) & \text{if} \ v|\infty \\\rule{0mm}{7mm}
\left\vert\mu_n(x)\right\vert_v^{{-{{r}}}} & \text{otherwise}\enspace,
\end{cases}
$$
where $o(1)\longrightarrow 0$ $(n\rightarrow +\infty)$ $($recall that $P_{l}(z)$ is of degree $rn$ in each variable $\alpha_i$, of degree $rmn+l$ in $z$ and constant in $t$, {and homogeneous of degree $rmn+l$ in $(\boldsymbol{\alpha},\beta)$}$)$.
\item[$(ii)$] 
The polynomial $P_{l,i,s}$ satisfies
$$
\Vert P_{l,i,s}( z)\Vert_v\leq   \begin{cases}
\exp\left(\dfrac{{n}d_v}{d}\left(\left[{\const}\right]+o(1)\right)\rule{0mm}{4mm}\right) & \text{if} \ v|\infty \\
\left\vert d_{brm(n+1)+a}\right\vert_v^{-s}\left\vert \mu_n(x)\right\vert_v^{-{{r}}} & \text{otherwise}\enspace,
\end{cases}
$$
{{
where $o(1)\longrightarrow 0$ $(n\rightarrow +\infty)$. 
}}
Also, $P_{l,i,s}(z)$ is of degree $\leq rmn+l$ in $z$ {and $\alpha_i$}, of degree $rn$ in each of the variables $\alpha_j$, {$j\neq i$},
$($recall that $\varphi_{\alpha_i,x,s}$ involves multiplication by $[\alpha_i]$$)$ {and homogeneous of degree $rmn+l$ in $(\boldsymbol\alpha,\beta)$.}
\item[$(iii)$] For any integer $k\geq 0$, the polynomial $\varphi_{\alpha_i,x,s}\circ[t^{k+n}](P_{l}(t))$ satisfies 
{\small{$$
\left\vert \varphi_{\alpha_i,x,s}\circ[t^{k+n}](P_{l}(t) {{)}}\right\vert_v\leq
\begin{cases}
\exp\left(\dfrac{nd_v}{d}\left(\left[{\const}\right]+o(1)\rule{0mm}{4mm} \right)\right)& \text{if} \ v|\infty \\
\left\vert d_{brm(n+1)+a}\right\vert_v^{-s}\left\vert\mu_n(x)\right\vert_v^{-{{r}}}\rule{0mm}{7mm} & \text{\kern-15ptotherwise}\enspace,
\end{cases}
$$}}
{{
where $o(1)\longrightarrow 0$ $(n\rightarrow +\infty)$. % \ \ \ \ {\green{$($referee comment $(9))$}}.}}
}}
By definition, it is a homogeneous polynomial in  just the variables $\boldsymbol{\alpha}$ of degree $rmn+l+k+n+1$.
\end{itemize}
\end{lemma}
\begin{proof} 
Set $B_{n,l}(\boldsymbol{y},t)=t^l\prod_{i=1}^my_i^{rn}$,
since $B_{n,l}$ is a monomial, its norm $\Vert B_{n,l}\Vert_v=1$. By definition, one has $P_{l}( z)=\ev_z\circ S^{(r)}_{n}\circ\Gamma(B_{n,l})$, and thus, by sub-multiplicatively of the endomorphism norm, 
{{
$$
\Vert P_{l}( z)\Vert_v
\leq 2^{\varepsilon_v rnm d_v/d}{{e^{\varepsilon_v o(n)}}}\cdot 
\begin{cases}
\left\vert \dfrac{(rmn+l+1+x)_n}{n!}\right\vert^{{{r}}}_v & \ \text{if} \ v: {\text{is archimedian or }  \text{non-archimedian} \ \& \  |x|_v>1}\\ 
1 & \ \text{otherwise}\enspace,
\end{cases}
$$
}} (one can choose $N=rn$ while using Lemma~\ref{norme} (iv) and $N=r(n+1)m+l$ for property (v) using $l\leq rm$, and note that the original polynomial is a constant in $z$ so the evaluation map is an isometry). 

\vspace{0.5\baselineskip}

In the ultrametric case, we have the claimed result
$$\left|\dfrac{(rmn+rm+1+x)_n}{n!}\right|_v\le |\mu_n(x)|^{-1}_v\enspace,$$
where $\mu_n(x):=b^n\prod_{\substack{q:\rm{prime} \\ q\mid b}}q^{\lceil n/(q-1) \rceil}$.
$(${\em confer} Lemma $5.6$ or originally \rm{\cite[Lemma $2.2$]{B}}$)$).

Left to prove is the archimedean case, we have (using $x<1$):
$$\dfrac{(rmn+rm+1+x)_n}{n!}\le \binom{(rm+1)(n+1)}{n}\enspace,$$
{and taking into account the standard Stirling formula
$$
\begin{array}{lcl}\displaystyle 
{(rm+1)(n+1)\choose n} & = & \displaystyle  \frac{((rm+1)(n+1))^{(rm+1)(n+1)}}{n^n((rm+1)(n+1)-n)^{(rm+1)(n+1)-n}}
\sqrt{\frac{(rm+1)(n+1)}{2\pi n((rm+1)(n+1)-n)}}(1+o(n))
 \\ & = & \displaystyle 
 c(r,m)n^{-1/2}\left(\frac{(rm+1)(n+1)}{n}\right)^n\left(\frac{1}{1-\frac{n}{(rm+1)(n+1)}}\right)^{(rm+1)(n+1)-n}(1+o(n))\enspace,
 \end{array}$$
 hence,
 $$\begin{array}{lcl}\displaystyle \frac{1}{n}\log
{(rm+1)(n+1)\choose n} & = & \log(rm+1)+rm\log\left(\frac{rm+1}{rm}\right)+o(1)
\enspace,
\end{array}$$
and putting these together, one gets
$$\Vert P_{l}(z)\Vert_v\leq \exp\left(\dfrac{nd_v}{d}\left[\const\right]+o(1)\right)\enspace,$$
where $o(1)\longrightarrow 0$ $(n\rightarrow +\infty)$.}
Now, by definition, $P_{l,i,s}(z)=\varphi_{\alpha_i,x,s}\circ \Theta(P_{l}(z))$ and $\varphi_{\alpha_i,x,s}=[\alpha_i]\circ \ev_{\alpha_i}\circ\inte^{(s)}$.
Using again Lemma~\ref{norme} (i), (ii), with $N=rn$, (iii) (vi) with $N=rm(n+1)$,  and since $rn=\exp(n\cdot o(1))$, one gets (ii),
where $o(1)\longrightarrow 0$ ($n\rightarrow +\infty$).

Finally, $\varphi_{\alpha_i,x,s}\circ[t^{k+n}](P_{l}( t))=[\alpha_i]\circ\ev_{\alpha_i}\circ\inte^{(s)}\circ
[t^{k+n}]\circ\ev_{z\rightarrow t}(P_{l}(z))$ where $\ev_{z\rightarrow t}$ is the map specializing $z$ into $t$. Again, using Lemma~\ref{norme}, one gets (iii).
\end{proof}
Recall that 
if $P$ is a homogeneous polynomial in some variables $y_i,i\in I$, for any point $\boldsymbol{\alpha}=(\alpha_i)_{i\in I}\in K^{\card(I)}$ where $I$ is any finite set, and $\Vert\cdot\Vert_v$ stands for the sup norm in $K_v^{\card(I)}$, {with $C_v(P)= (\deg(P)+1)^{\frac{\varepsilon_vd_v(\card(I))}{d}}$}
\begin{equation}\label{estimhomo}
\vert P(\boldsymbol{\alpha})\vert_v\leq C_v(P) \Vert P\Vert_v\cdot {\Vert {\boldsymbol{\alpha}\Vert_v^{\deg(P)}}}\enspace.
\end{equation}
So, the preceding lemma yields trivially estimates for the $v$-adic norm of the above given polynomials.
\begin{remark}(i) It may be worthwhile to note that it is also possible to estimate with respect to {$\Vert (\alpha_i-\beta)\Vert_v$ instead of $\Vert\alpha_i\Vert_v$} which can be useful in certain circumstances (saving of the archimedean error term $rmn\log(2)$, useful if the local heights of $\beta,\alpha_i$ are not too far apart).
\end{remark}
\begin{lemma} \label{upper jyouyonew}
Let $n$ be a positive integer, {{$x=a/b\in \Q\cap [0,1)$}} and $\beta\in K$ with $\Vert \boldsymbol{\alpha}\Vert_v<\vert \beta\vert_v$. 
Then we have for all $i\leq m,l\leq rm,s\leq r$,
\begin{align*}
|R_{l,i,s}(\beta)|_v &\le \Vert \boldsymbol{\alpha}\Vert_v^{rm(n+1)}
\cdot \left(\frac{\Vert\boldsymbol{\alpha}\Vert_v}{\vert \beta\vert_v}\right)^{n+1}\cdot\left(\frac{\varepsilon_v\vert \beta\vert_v}{\vert \beta\vert_v-\Vert \boldsymbol{\alpha}\Vert_v}+(1-\varepsilon_v)\right)\\
&\times 
\begin{cases}
\exp\left(n{{\dfrac{d_v}{d}}}\left[{\const}\right] +o(1)\rule{0mm}{4mm}\right) & \text{if} \ v|\infty \\
\left\vert d_{brm(n+1)+a}\right\vert_v^{-s}\left\vert\mu_n(x)\right\vert_v^{-{{r}}}\rule{0mm}{7mm} & \text{otherwise}\enspace,
\end{cases}
\enspace
\end{align*} 
where $o(1)\longrightarrow 0$ $(n\rightarrow +\infty)$.
\end{lemma} 
\begin{proof}
By the definition of $P_{l}(z)$, as formal power series, we have
$$
R_{l,i,s}(z)=\sum_{k=0}^{\infty}\frac{\varphi_{\alpha_i,x,s}(t^{k+n}P_{l}(t))}{z^{k+n+1}} \enspace, 
$$
using the triangle inequality, the fact that $l\leq rm$ and Lemma~\ref{majonorme}, (iii), together with inequality~(\ref{estimhomo})
\begin{align*}
\vert R_{l,i,s}(\beta)\vert_v&\leq \Vert \boldsymbol{\alpha}\Vert_v^{rm(n+1)}\sum_{k=0}^{\infty}\left(\frac{\Vert\boldsymbol{\alpha}\Vert_v}{\vert \beta\vert_v}\right)^{n+1+k}\\
&\times 
\begin{cases}
\exp\left(n{{\dfrac{d_v}{d}}}\left[{\const} +o(1)\rule{0mm}{4mm}\right]\right) & \text{if} \ v|\infty \\
\left\vert d_{brm(n+1)+a}\right\vert_v^{-s}\left\vert\mu_n(x)\right\vert_v^{-{{r}}}\rule{0mm}{7mm} & \text{otherwise}\enspace,
\end{cases}
\end{align*}
and the lemma follows using geometric series summation (here, $o(1)\longrightarrow 0$ ($n\rightarrow +\infty$)).
\end{proof}
\subsection{Proof of Theorem $\ref{Lerch}$}
Before starting to prove Theorem $\ref{Lerch}$, we introduce a criterion to obtain a linear independence measure of numbers.
\begin{proposition} \label{critere version II} 
Let $K$ be a number field and $v_0$ a place of $K$. {{We denote the completion of $K$ with respect to the fixed embedding $\iota_{v_0}$ by $K_{v_0}$.}}
Let $m\in\N$ and ${\boldsymbol{0}\neq} \boldsymbol{\theta}=(\theta_0,\ldots,\theta_m)\in K_{v_0}^{{{m+1}}}$. 
Suppose there exist a sequence of matrices $$({{\mathrm{M}}}_n)_n=\left(\left(a^{(n)}_{l,j}\right)_{0\le l,j\le m}\right)_{n\in \N}\subset  {\rm{GL}}_{m+1}(K)\enspace,$$ positive real numbers
${{U}}, \mathbb{A}$, $\{F_v:\N\rightarrow \R_{\ge0} \}_{v\in {{\mathfrak{M}}}_K}$ with
\begin{align}
&\max_{0\le l \le m}|a^{(n)}_{l,0}|_{v_0}\le e^{{{U}}n+o(n)}\enspace, \label{upper Anr0}\\
&\max_{\substack{0\le l \le m \\ 1\le j \le m}}\left\vert a^{(n)}_{l,0}\theta_j-a^{(n)}_{l,j}{{\theta_0}}\right\vert_{v_0}\leq e^{-\mathbb{A} n+o(n)}\enspace, \label{upper Rrj}\\
&\left\Vert {{\mathrm{M}}}_n\right\Vert_v \leq  e^{F_v(n)} \ \ (v\in {{\mathfrak{M}}}_K)\enspace, \label{upper Anrj}
\end{align}
for $n\in \N$.  
We assume ${\displaystyle{\lim_{n\to \infty}}}1/n {\displaystyle{\sum_{v\neq v_0}}} F_v(n)=:\mathbb{B}$ exists. 
Put
\begin{align*}
&V:=\mathbb{A}-\mathbb{B} \enspace,
\end{align*} 
and assume $V>0$.
Then for any $0<\varepsilon<V$, there exists {a} constant $H_0=H_0(\varepsilon)>0$ depending on $\varepsilon$ and the given data such that the following property holds.
For any ${{\boldsymbol{\lambda}:=(\lambda_0,\ldots,\lambda_m)}} \in K^{m+1} \setminus \{ \bold{0} \}$ satisfying $H_0\le {\mathrm{H}}({{\boldsymbol{\lambda}}})$, we have
\begin{align*}
\left|\sum_{i=0}^m{{\lambda_i}}\theta_i\right|_{v_0}>C(\varepsilon) {{{\mathrm{H}}_{v_0}(\boldsymbol{\lambda}}}) {\mathrm{H}}({{\boldsymbol{\lambda}}})^{-\mu(\varepsilon)}\enspace,
\end{align*}
where 
$$
\mu(\varepsilon):=
\dfrac{\mathbb{A}+{{U}}}{V-\varepsilon}\enspace \mbox{ \it and }\hspace{7pt}   C{{(\varepsilon)}}:=\exp\left[-{{\left(\frac{\log(2)}{V-\varepsilon}+1\right)}}(\mathbb{A}+{{U}})\right]\enspace.
$$
\end{proposition}
\begin{proof}
We may assume $\theta_0=1$.
Let ${{\boldsymbol{\lambda}:=^t\kern-5pt(\lambda_0,\ldots,\lambda_m)}}\in K^{m+1}\setminus \{\bold{0} \}$. 
Define 
$\Lambda({{\boldsymbol{\lambda}}},\boldsymbol{\theta})=\sum_{i=0}^m{{\lambda_i}}\theta_i$.
Since ${{\mathrm{M}}}_n$ is invertible and ${{\boldsymbol{\lambda}}}\neq \bold{0}$, some coordinate (say the $l_n$-th coordinate) $B_{l_n}:=\sum_{j=0}^{m}a^{(n)}_{l_n,j}{{\lambda_j}}\neq 0$ of ${{\mathrm{M}}}_n{{\boldsymbol{\lambda}}}$ does not vanish.
We put $r^{(n)}_{l,j}:=a^{(n)}_{l,0}\theta_j-a^{(n)}_{l,j}$ for $1\le j \le m$ and $0\le l \le m$.
Then by definition, we obtain
$$
a^{(n)}_{l_n,0}\Lambda({{\boldsymbol{\lambda}}}, \boldsymbol{\theta})=B_{l_n}+\sum_{j=1}^{m}r^{(n)}_{l_n,j}{{\lambda_j}}\enspace.
$$
Using the product formula for $B_{l_n}\in K\setminus\{0\}$, we have
\begin{equation} \label{upper infty}
0=\sum_{v\neq v_0}\log\left\vert B_{l_n}\right\vert_v+\log\left\vert a^{(n)}_{l_n,0}\Lambda({{\boldsymbol{\lambda}}},\boldsymbol{\theta})-\sum_{j=1}^{m}r^{(n)}_{l_n,j}{{\lambda_j}}\right\vert_{v_0}\enspace.
\end{equation}
We first evaluate using the inequalities $(\ref{upper Anrj})$ and $(\ref{upper Rrj})$
{{
\begin{align*}
&\log\left\vert B_{l_n} \right\vert_{v}\leq {\mathrm{h}}_v({{\boldsymbol{\lambda}}})+{{F_v}}(n)+{\varepsilon_v \log(m+1)}\ \ \text{for} \ v\in {{\mathfrak{M}}}_K\enspace,\\
&\log\left\vert\sum_{j=1}^{m}r^{(n)}_{l_n,j}{{\lambda_j}}\right\vert_{v_0}\leq {\mathrm{h}}_{v_0}({{\boldsymbol{\lambda}}})-\mathbb{A}n+o(n) +\varepsilon_{v_0}\log(m)\enspace.
\end{align*}
}}
We now fix a positive number $\varepsilon$ with $V>\varepsilon$, and assume that $n$ is large enough so that  {all $\left\vert o(n)\right\vert\leq\varepsilon n/4$},  
$[K:\Q]\cdot {\rm{log}}(m+1)\le \varepsilon n /4$, $0<-\log(2)+n(V-\varepsilon)$ and $\left\vert \displaystyle{\sum_{v\neq v_0}}F_v(n)-\mathbb{B}n\right\vert\leq \varepsilon n/4$. 
We now fix such $n^*$ and take ${\mathrm{h}}_0=-\log(2)+n^*(V-\varepsilon)$. 
Take $\boldsymbol{\lambda}\in K^{m+1}\setminus \{\bold{0}\}$ with ${\mathrm{h}}(\boldsymbol{\lambda})\ge {\mathrm{h}_0}$. Let $n=n(\boldsymbol{\lambda})$ be the minimal positive integer with 
${\mathrm{h}}({{\boldsymbol{\lambda}}})-n(V-\varepsilon)\leq -{\rm{log}}(2)$. 
Note that $n\ge n^*$ and $n<\frac{{\mathrm{h}}({{\boldsymbol{\lambda}}})}{V-\varepsilon}+\frac{\log(2)}{V-\varepsilon}+1$. 
{With these conventions,
$$-\sum_{v\neq v_0}\log\vert B_{l_n}\vert_v\geq -\sum_{v\neq v_0} \mathrm{h}_v(\boldsymbol{\lambda})-\sum_{v\neq v_0}F_v(n)-{{\sum_{v\neq v_0}}}\varepsilon_v\log(m+1)\geq -\sum_{v\neq v_0}\mathrm{h}_{v}(\boldsymbol{\lambda})-\mathbb{B}n-\varepsilon n/2\enspace.$$
Plugging in this inequality in~(\ref{upper infty}),
$${\small{{{-\sum_{v\neq v_0}\mathrm{h}_v(\boldsymbol{\lambda})}}-\mathbb{B}n-\varepsilon n/2\leq 
\log\left\vert a^{(n)}_{l_n,0}\Lambda({{\boldsymbol{\lambda}}},\boldsymbol{\theta})-\sum_{j=1}^{m}r^{(n)}_{l_n,j}{{\lambda_j}}\right\vert_{v_0}\leq \log\left( \left\vert a^{(n)}_{l_n,0}\Lambda({{\boldsymbol{\lambda}}},\boldsymbol{\theta})\right\vert_{v_0}+\left\vert\sum_{j=1}^{m}r^{(n)}_{l_n,j}{{\lambda_j}}\right\vert_{v_0}\right)}}\enspace.$$

One notes  ${\mathrm{h}}({{\boldsymbol{\lambda}}})-n(V-\varepsilon)\leq -{\rm{log}}(2)$ translates into
$$\mathrm{h}_{v_0}(\boldsymbol{\lambda})-\mathbb{A}n+\varepsilon_{v_0}\log(m)\leq -\sum_{v\neq v_0}\mathrm{h}_v(\boldsymbol{\lambda})-\mathbb{B}n-{{3}}n\varepsilon/4-\log(2)\enspace.$$

This implies $\log\left\vert\sum_{j=1}^{m}r_{l_n,j}^{(n)}\lambda_j\right\vert\leq -\sum_{v\neq v_0}\mathrm{h}_v(\boldsymbol{\lambda})-\mathbb{B}n-{{3}}n\varepsilon/4-\log(2)$. We now use the elementary inequality valid for any positive real numbers {{$x,y$ with}} $\log(x+y)\geq a$ and $\log(y)\leq a-\log(2)$ implies $\log(x)\geq a-\log(2)$ and deduce
$$-\sum_{v\neq v_0}\mathrm{h}_v(\boldsymbol{\lambda})-\mathbb{B}n-{{3n\varepsilon/4}}-\log(2)\leq \log\left\vert a_{l_n,0}^{(n)}\Lambda(\boldsymbol{\lambda},\boldsymbol{\theta})\right\vert_{v_0}\leq Un+\varepsilon n/4+\log\left\vert\Lambda(\boldsymbol{\lambda},\boldsymbol{\theta})\right\vert_{v_0} \enspace.$$

Now,
$$
\begin{array}{lcl} \displaystyle 
 \sum_{v\neq v_0}\mathrm{h}_v(\boldsymbol{\lambda})+\mathbb{B}n+Un+{{n\varepsilon}}+\log(2)
& = & \displaystyle  \mathrm{h}(\boldsymbol{\lambda})-\mathrm{h}_{v_0}(\boldsymbol{\lambda})
+(\mathbb{A}+U)n-n(V-\varepsilon)-\varepsilon n+\log(2)
\\ & \leq  &  \displaystyle  \frac{(\mathbb{A}+U)\mathrm{h}(\boldsymbol{\lambda})}{V-\varepsilon}
+\left(\frac{\log(2)}{V-\varepsilon}+1\right)\left(\mathbb{A}+U\right)- \mathrm{h}_{v_0}(\boldsymbol{\lambda})\enspace.
\end{array}$$
All in all,
$$\log\left\vert\Lambda(\boldsymbol{\lambda},\boldsymbol{\theta})\right\vert_{v_0} 
\geq \mathrm{h}_{v_0}(\boldsymbol{\lambda})-\mu(\varepsilon)\mathrm{h}(\boldsymbol{\lambda})-\log C(\varepsilon)\enspace.$$
This completes the proof of the proposition.}
\end{proof}
To prove Theorem $\ref{Lerch}$, we show the following more precise theorem.

\begin{theorem} \label{Lerch 2}{
We use the same notations as in Theorem $\ref{Lerch}$. For any place  $v\in {{\mathfrak{M}}}_K$, we define the constants
\begin{align*}
c(x,v)=\varepsilon_{v}{{\frac{d_v}{d}}}\left[\const\right]-(1-\varepsilon_{v})r\log\vert\mu(x)\vert_{v} \enspace.
\end{align*}
{{We also define
\begin{align*}
\mathbb{A}(\boldsymbol{\alpha},\beta)&=\mathbb{A}_{v_0}(\boldsymbol{\alpha},\beta)=\log\vert \beta \vert_{v_0}-(rm+1)\log\Vert \boldsymbol{\alpha}\Vert_{v_0}-c(x,v_0)+{{r(1-\varepsilon_{v_0}){\lim_{n\rightarrow\infty}\frac{{\rm{log}}\vert d_{brm(n+1)+a}\vert_{v_0}}{n}}}}
\enspace,\\U(\boldsymbol{\alpha},\beta)&=U_{v_0}(\boldsymbol{\alpha},\beta)={rm{\mathrm{h}}_{v_0}({\boldsymbol\alpha},\beta)+c(x,v_0)\enspace,\kern-3pt}
\end{align*}
and
$$\begin{array}{lcl}\displaystyle 
V(\boldsymbol{\alpha},\beta)&= V_{v_0}(\boldsymbol{\alpha},\beta)  = &\displaystyle  \log\vert\beta\vert_{v_0}-rm{\mathrm{h}}(\boldsymbol{\alpha},\beta)-{{(rm+1)}}\log\Vert \boldsymbol{\alpha}\Vert_{v_0}+rm\log\Vert(\boldsymbol{\alpha},\beta)\Vert_{v_0}\\ &{\hspace{17mm}}-&\displaystyle\left[\const\right]-r\log\mu(x)-br^2m\enspace\rule{0mm}{6mm},
\end{array}$$
where $b={\rm{den}}(x)$.}}

\

{{Let $v_0$ be a place in $\mathfrak{M}_K$,  either archimedean or non-archimedean,  such that
$V_{v_0}(\boldsymbol{\alpha},\beta)>0$.
Then the functions $\Phi_s(x,z), \,\,1\leq s \leq r$ converge around $\alpha_j/\beta$ in $K_{v_0}$,  $1\leq j \leq m$ and}}
for any positive number $\varepsilon$ with $\varepsilon<V(\boldsymbol{\alpha},\beta)$, there exists an effectively computable positive number $H_0$ depending on $\varepsilon$ and the given data such that the following property holds.
For any ${{\boldsymbol{\lambda}}}:=({{\lambda_0}},{{\lambda_{i,s}}})_{\substack{1\le i \le m \\ 1\le s \le r}} \in K^{rm+1} \setminus \{ \bold{0} \}$ satisfying $H_0\le {\mathrm{H}}({{\boldsymbol{\lambda}}})$, then we have
\begin{align*}
\left|{{\lambda_0}}+\sum_{i=1}^m\sum_{s=1}^{r}{{\lambda_{i,s}}}\Phi_{s}(x,\alpha_i/\beta)\right|_{v_0}>C(\boldsymbol{\alpha},\beta,\varepsilon){\mathrm{H}}_{v_0}({{\boldsymbol{\lambda}}}) {\mathrm{H}}({{\boldsymbol{\lambda}}})^{-\mu(\boldsymbol{\alpha},\beta,\varepsilon)}\enspace,
\end{align*}
where 
\begin{align*}
\mu(\boldsymbol{\alpha},\beta,\varepsilon):=
\dfrac{\mathbb{A}(\boldsymbol{\alpha},\beta)+{{U}}(\boldsymbol{\alpha},\beta)}{V(\boldsymbol{\alpha},\beta)-\varepsilon} \enspace \mbox{and} \enspace C(\boldsymbol{\alpha},\beta,\varepsilon)=\exp\left(-{{\left(\frac{\log(2)}{V(\boldsymbol{\alpha},\beta)-\varepsilon}+1\right)}}(\mathbb{A}(\boldsymbol{\alpha},\beta)+{{U}}(\boldsymbol{\alpha},\beta)\right)\enspace.
\end{align*}}
\end{theorem}
\begin{proof}
By Proposition $\ref{non zero det}$, the matrix ${{\mathrm{M}}}_n=\begin{pmatrix}
P_{l}(\beta)\\
P_{l,i,s}(\beta)
\end{pmatrix}$ 
with entries in $K$ is invertible. We can apply Proposition~\ref{critere version II}. 
We have by Lemma~\ref{majonorme}, (i) together with inequality~(\ref{estimhomo}) and Lemma~\ref{denominator 1}
$$\begin{array}{lcl}\displaystyle
\log\Vert  P_{l}(\beta)\Vert_{v} & \leq& \displaystyle \varepsilon_v\left(n{\frac{d_v}{d}}{\left[\const\right]}+o(1)\rule{0mm}{4mm}\right)\\
& + &(1-\varepsilon_v)\log\left\vert\mu_n(x)\right\vert_v^{-{{r}}}+\displaystyle {(rmn+l){\mathrm{h}}_{v}({\boldsymbol \alpha},\beta)}\rule{0mm}{6mm}
\\ & \leq   &\rule{0mm}{6mm}  \displaystyle n\left({rm{\mathrm{h}}_{v}({\boldsymbol\alpha},\beta)}+c(x,v)\right)+o(n)
\\ & = &\rule{0mm}{6mm} \displaystyle {{U}}(\boldsymbol{\alpha},\beta)n+o(n)\enspace,
\end{array}$$
where $o(1)\longrightarrow 0$ ($n\rightarrow +\infty$). Similarly, using this time Lemma~\ref{majonorme} (ii), 
$$\begin{array}{lcl}\displaystyle
\log\Vert  P_{l,i,s}(\beta)\Vert_{v} & \leq& \displaystyle \varepsilon_v\left(n{{\frac{d_v}{d}}}{\left[\const\right]}+o(1)\rule{0mm}{4mm}\right)\\
& + & (1-\varepsilon_v)\log\left\vert\mu_n(x)\right\vert_v^{-{{r}}}+\rule{0mm}{6mm}\displaystyle{+(rmn+l){\mathrm{h}}_{v}({\boldsymbol\alpha},\beta)}-s(1-\varepsilon_v)\log\vert d_{brm(n+1)+a}\vert_v
\\ & \leq   &  \rule{0mm}{6mm}\displaystyle n\left({rm{\mathrm{h}}_{v}({\boldsymbol\alpha},\beta)}+c(x,v)\right)+f_v(n)\enspace,
\end{array}$$
where
\begin{align*}
f_v:\N\rightarrow \R_{\ge0}; \ \ n\longmapsto {lh_v({\boldsymbol\alpha},\beta)}-r(1-\varepsilon_v)\log\left(\vert d_{brm(n+1)+a}\vert_v\right)\enspace.
\end{align*}
We define $$\displaystyle {{F_v}}(\boldsymbol{\alpha},\beta),\N\rightarrow \R_{\ge0}; \ n\mapsto n\left({rm{\mathrm{h}}_{v}({\boldsymbol\alpha},\beta)+c(x,v)}\right)+f_v(n)\enspace.$$
Since on the other hand, Lemma~\ref{upper jyouyonew} ensures 
{\small{\begin{align*}
-\log\vert  R_{l,i,s}( \beta)\vert_{v_0}&\leq n\log\vert \beta \vert_{v_0}-(rm+1)n\log\Vert \boldsymbol{\alpha}\Vert_{v_0}\\
&-\kern-2pt
\begin{cases}
\dfrac{nd_v}{d}{\const}+o(n) & \kern-3pt \text{if}\    v_0\in \mathfrak{M}^{\infty}_K \\
\left\vert d_{brm(n+1)+a}\right\vert_{v_0}^{-r}\left\vert\mu_n(x)\right\vert_{v_0}^{-r}  +o(n) &\kern-3pt \text{if} \ v_0 \in \mathfrak{M}^f_K
\end{cases}\\
&=n\log\vert \beta \vert_{v_0}-(rm+1)n\log\Vert \boldsymbol{\alpha}\Vert_{v_0}-c(x,v_0)n+r(1-\varepsilon_{v_0}){\rm{log}}\vert d_{brm(n+1)+a}\vert_{v_0}\\
&=\mathbb{A}(\boldsymbol{\alpha},\beta)n+o(n)\enspace,
\end{align*}}}
we can apply Proposition~\ref{critere version II} for $\{\theta_{i,s}:=\Phi_{s}(x,\alpha_i/\beta)\}_{\substack{1\le i \le m \\ 1\le s \le r}}$ and the above data, 
we obtain the assertions of Theorem $\ref{Lerch 2}$. 
Using {{Lemma $\ref{denominator 1}$ $(iv)$}}, we have ${\displaystyle{\lim_{n\to \infty}}}1/n{\displaystyle{\sum_{v\in \mathfrak{M}_K}}}f_v(n)=br^2m$ and
$$
\sum_{v\in \mathfrak{M}_K}c(x,v)={\const}+{\rm{log}}\mu(x)\enspace,
$$
we conclude
{$$\begin{array}{lcl}\displaystyle
\mathbb{A}(\boldsymbol{\alpha},\beta)-\lim_{n\to \infty}\dfrac{1}{n}\sum_{v\neq v_0}{{F_v}}(\boldsymbol{\alpha},\beta)(n) & = & \displaystyle {\rm{log}}|\beta|_{v_0}-rm{\mathrm{h}}({\boldsymbol\alpha},\beta)+rm\log\Vert({\boldsymbol\alpha},\beta)\Vert_{v_0}\\ & 
- & \displaystyle(rm+1){\rm{log}}||\boldsymbol{\alpha}||_{v_0}
-\sum_{{v}}c(x,v){{-br^2m}} = V(\boldsymbol{\alpha},\beta)\enspace. %\nonumber\enspace.
\end{array}$$}
\end{proof} 

{{\noindent \bf End of the proof of Theorem~\ref{Lerch}.}~
\\
As was already noted, this is an immediate consequence of Theorem~\ref{Lerch 2}~;  to be more precise, choose some $\boldsymbol{\lambda}=(\lambda_0,\lambda_{1,1}\ldots,\lambda_{r,m})\in K^{rm+1}\setminus\{\bold{0}\}$ such that $\lambda_0+\sum_{i=1}^{m}\sum_{s=1}^{r}\lambda_{i,s}\Phi_{s}(x,\alpha_i/\beta)=0$. 

If $H(\boldsymbol{\lambda})\geq H_0$ (where $H_0$ is as in Theorem~\ref{Lerch 2}), there is nothing more to prove. Otherwise, let $q>0$ be a rational integer such that $H(q\boldsymbol{\lambda})\geq H_0$. Then Theorem~\ref{Lerch 2} ensures that
$$q\left(\lambda_0+\sum_{i=1}^{m}\sum_{s=1}^{r}\lambda_{i,s}\Phi_{s}(x,\alpha_i/\beta)\right)\neq 0\enspace.$$ This completes the proof of Theorem~\ref{Lerch}. \qed}

{\subsection{Comparison with the main result of \cite{DHK2}}
{{
We recall the main Theorem in \cite{DHK2}.
Under the same notations as in Theorem $\ref{Lerch 2}$, we fix an embedding $\iota:K\hookrightarrow \C$ and denote the corresponding archimedean place by $v_0$. %and $\sigma(\alpha)=\alpha$ for $\alpha\in K$.
We define 
\begin{align*} 
&D(\boldsymbol{\alpha},\beta):={\rm{den}}(\alpha_1,\ldots,\alpha_m,\beta)\enspace,\\
&\mathbb{A}(\boldsymbol{\alpha},\beta,x):={\rm{log\,}}|\beta|-(rm+1){\rm{log\,}}\max_i(|\alpha_i|)-\{rm({\rm{log\,}}D(\boldsymbol{\alpha},\beta)+r[b+{\rm{log\,}}(5/2)])+r({\rm{log\,}}3+{\rm{log\,}}\mu(x))\}\enspace,\\
&\mathcal{A}^{(g)}(\boldsymbol{\alpha},\beta,x):=rm\left({\rm{log\,}}D(\boldsymbol{\alpha},\beta)+{\rm{log\,}}\max(1, \min(|\alpha^{(g)}_i|)^{-1}\cdot|\beta^{(g)}|)+r[b-{\rm{log\,}}2]\right)\\
&\ \ \ \ \ \ \ \ \ \ +r\left({\rm{log\,}}\mu(x)+\sum_{i=1}^m{\rm{log\,}}(2^r|\alpha^{(g)}_i|+3^r\max_{1\le i'\le m}(|\alpha^{(g)}_{i'}|,|\beta^{(g)}|))+{\rm{log\,}}3 \right)
 \ \ \ \ \text{for} \ \  1\le g\le d \enspace,
\end{align*}
where $\alpha^{(g)}_i$ (resp. $\beta^{(g)}$) is a conjugate of $\alpha_i$ (resp. $\beta$) for $1\le g \le d$ with $\alpha^{(1)}_i=\alpha_i$ and $\alpha^{(2)}_i$ is the complex conjugate of $\alpha_i$ if $K_{v_0}=\C$ and
$$
{M}(\boldsymbol{\alpha},\beta,x):=\mathbb{A}(\boldsymbol{\alpha},\beta,x)+\mathcal{A}^{(1)}(\boldsymbol{\alpha},\beta,x)-\dfrac{\sum_{g=1}^{d}\mathcal{A}^{(g)}(\boldsymbol{\alpha},\beta,x)}{d_{v_0}}\enspace.
$$
\begin{theorem} $(${\rm{\cite[Theorem $2.1$]{DHK2}}}$)$ \label{Moscow}
Assume ${M}(\boldsymbol{\alpha},\beta,x)>0$.
Then the $rm+1$ numbers~$:$
$$1,\Phi_1(x,\alpha_1/\beta),\ldots,\Phi_r(x,\alpha_1/\beta),\ldots, \Phi_1(x,\alpha_m/\beta),\ldots,\Phi_r(x,\alpha_m/\beta)\enspace,$$ are linearly independent over $K$.
\end{theorem}
}}
We are now in a position to compare this result with our Theorem~\ref{Lerch}
\begin{proposition}\label{compamoscow} 
Assume $r\geq 7$ or $r=5,6$ and $m\geq 2$ or $r=4, m\geq 3$, $r=3,m\geq 6$ $r=2,m\geq 36$, then
$$V(\boldsymbol{\alpha},\beta)\geq \frac{d_{v_0}}{d}M(\boldsymbol{\alpha},\beta,x)\enspace,$$
in other words, the present result $($Theorem $\ref{Lerch})$ is stronger than Theorem $\ref{Moscow}$.
\end{proposition}
\begin{proof}
We put $u(m,r,x)=r^2m[b-\log2]+r\log(3)+r\log\mu(x)$ and estimate
\begin{align*}
&\displaystyle\frac{1}{d}\sum_{g=1}^d\mathcal{A}^{(g)}(\boldsymbol{\alpha},\beta,x) 
\geq \displaystyle r mh_{\infty}(\boldsymbol{\alpha},\beta)+rm\log D(\boldsymbol{\alpha},\beta)+rmh_{\infty}(1,\boldsymbol{\alpha}^{-1}\beta)+rm\log( 3^r)+u(r,m,x)\enspace.
\end{align*}
We also put $v(m,r,x)=\frac{d_{v_0}}{d}\left(rm\left(r[b+{\rm{log\,}}(5/2)])+r({\rm{log\,}}3+{\rm{log\,}}\mu(x))\right)\right)$ and estimate
\begin{align*}
\displaystyle \frac{d_{v_0}}{d}\left(\mathbb{A}(\boldsymbol{\alpha},\beta,x)+\mathcal{A}^{(1)}(\boldsymbol{\alpha},\beta,x)\right)&\leq 
\displaystyle\log|\beta|_{v_0}-(rm+1)\log\Vert\boldsymbol{\alpha}\Vert_{v_0}+rm\log\Vert (1,\boldsymbol{\alpha}^{-1}\beta)\Vert_{v_0}\\
&+rm\log\Vert (\boldsymbol{\alpha},\beta)\Vert_{v_0}- \displaystyle v(m,r,x)+ \displaystyle \frac{d_{v_0}}{d} u(m,r,b,x)+\frac{d_{v_0}}{d}rm\log(2\cdot 3^r)\enspace.
\end{align*}
Then we get
\begin{align*}
\frac{d_{v_0}}{d}M(\boldsymbol{\alpha},\beta) & \leq
\log|\beta|_{v_0}-(rm+1)\log\Vert\boldsymbol{\alpha}\Vert_{v_0}-rm \log D(\boldsymbol{\alpha},\beta)+\displaystyle w(m,r,x)\\
&-rm\left(h_{\infty}(\boldsymbol{\alpha},\beta)-\log\Vert (\boldsymbol{\alpha},\beta)\Vert_{v_0} +h_{\infty}(1,\boldsymbol{\alpha}^{-1}\beta)
-\log\Vert (1,\boldsymbol{\alpha}^{-1}\beta)\Vert_{v_0}\right)\enspace,
\end{align*}
{{where}}
\begin{align*}
&w(m,r,x)=-v(m,r,x)+\frac{d_0}{d} u(m,r,x)+\frac{d_0}{d}rm\log(2\cdot 3^r)-rm\log( 3^r)-u(r,m,x)\enspace,\\
&c(r,m,x)=\displaystyle\const+r\log\mu(x)+br^2m\enspace\rule{0mm}{6mm}\enspace,\\
&H(r,m,x)=-c(r,m,x) -w(r,m,x)\enspace.
\end{align*}
Then we obtain 
{\small{$$V(\boldsymbol{\alpha},\beta)-\frac{d_{v_0}}{d}M(\boldsymbol{\alpha},\beta,x)\geq rm\left(h_{\infty}(\boldsymbol{\alpha},\beta)+ \log D(\boldsymbol{\alpha},\beta)-h(\boldsymbol{\alpha},\beta)+h_{\infty}(1,\boldsymbol{\alpha}^{-1}\beta)-\log\Vert (1,\boldsymbol{\alpha}^{-1}\beta)\Vert_{v_0}\right)+H(r,m,x)\enspace.$$}}
Since one has trivially
$$h_{\infty}(\boldsymbol{\alpha},\beta)+ \log D(\boldsymbol{\alpha},\beta)-h(\boldsymbol{\alpha},\beta)>0\enspace; h_{\infty}(1,\boldsymbol{\alpha}^{-1}\beta)-\log\Vert (1,\boldsymbol{\alpha}^{-1}\beta)\Vert_{v_0}>0\enspace,$$
proving the proposition boils down to proving $H(r,m,x)>0$.
We first assume $r\geq 2$  then, the right hand side is increasing in $d_{v_0}/d$ and we can assume $d_{v_0}/d=0$.
We get,
$$H(m,r,x)\geq r^2m\log(3/2)+r\log(3)-rm\log(2)-r\log(rm+1)-r^2m\log\left(\frac{rm+1}{rm}\right)\enspace.$$
One then shows (elementary but tedious computation) that the right hand side is increasing in $m$ for $r\geq 4$ and get the proposition by checking the smallest value of $m$ for which it gets positive. 
For $m=2,3$, one shows that the right hand side is increasing in $m$ for $m$ large enough and dispose of the (finitely many) remaining case by case by case testing. 
This completes the proof of the proposition.
\end{proof}
\begin{remark}
\label{diago}
Using the facts that the maps $S_{n,x}$ and $\inte_x$ act diagonally on the canonical basis of the polynomials, one can get better bounds for $\Vert P_l\Vert_v$, $\Vert P_{l,i,s}\Vert_v$ and $|R_{l,i,s}(\beta)|_v$, by replacing the sub\-multiplicativity of the norms by more precise estimates for each individual eigenvalue in the proof of Lemma~\ref{majonorme}. This refinement would be enough to prove Theorem~\ref{Lerch} is stronger than Theorem~\ref{Moscow} as soon as $r\geq 2$. However, it would not be enough to say anything meaningful for\footnote{Indeed, it can be shown that when $r=1$ neither theorem implies the other.} $r=1$. Our intent {in the present paper is} to provide reasonably good analytic estimates (see subsection~\ref{comphata} below) with minimal technical complications and of course a better dependence in the main parameters $\beta$ and $\boldsymbol{\alpha}$.
\end{remark}

\subsection{Comparison with the bounds of M. Hata}
\label{comphata}
{\begin{remark} 
We compare Theorem $\ref{Lerch 2}$ in the case of $m=1$ and $x=0$, and $K=\qu$ with the result of M. Hata \cite{Ha}. 
{We set $E(r)=\frac{e^{r^2}r^r(e^{-r}+1)^{r+1}}{(r+1)^{r+1}}$.}
Below we cite the corresponding theorem due to Hata\footnote{We present his simplified version, when $b>0$. Hata's bound are slightly better when $b<0$, however it does not quite change our discussion.}~:\\
\textbf{Theorem of Hata} {\rm{\cite[Corollary $2.2$]{Ha}}}\\
{\it Let $r$ be a natural number and $b$ a non-zero integer. 
Put 
$$\tilde{V}(b):=\log|b|-\log (E_r)\enspace.$$
Assume $\tilde{V}(b)>0$. 
Then the $r+1$ real numbers $1,{\rm{Li}}_1(1/b),\ldots, {\rm{Li}}_r(1/b)$ are linearly independent over $\Q$.}

{Let $b>0$ be an integer. We consider the case of $m=1$, $\alpha_1=1$, $x=0$, $K=\Q$ and $v_0=\infty$. Then, our Theorem~\ref{Lerch 2} asserts that setting 
$$V_{\infty}(1,b)=  \log\vert b \vert_{\infty}-r^2-{{r\left(\log(2)+\log(r+1)+r\log\left(\dfrac{r+1}{r}\right)\right)}}\enspace,$$
the numbers $1,{\rm{Li}}_1(1/b),\ldots, {\rm{Li}}_r(1/b)$ are linearly independent over $\Q$ provided $V_{\infty}(1,b)>0$.
Hence, to compare both results, one needs only to compare $\tilde{V}(b)$ and $V_{\infty}(1,b)$.
One easily checks that for all $r\geq 1$, 
$$\log(E_r)\leq  r^2+{{r\left(\log(2)+\log(r+1)+r\log\left(\dfrac{r+1}{r}\right)\right)}}\enspace,$$
Hence, when specialized to this particular case, Hata's result is always stronger then ours. Interestingly, they are not that far apart for large $r$, indeed, 
{
$$Q(r)=\frac{\left(r^2+{{r\left(\log(2)+\log(r+1)+r\log\left(\dfrac{r+1}{r}\right)\right)}}\right)-\log(E_r)}{\log(E_r)}=O\left(\frac{\log(r)}{r}\right)\enspace.$$}}

For convenience, we record some special values

{{\small{$$\begin{array}{|c|c|c|c|c|c|c|c|}\hline
& r=1 & r=2 & r=3 & r=4 & r=10 & r=100 & r=1000\\\hline
\log(E_r)& 0.2402& 2.4712 & 6.945 & 13.5887 & 96.6495 & 9994.38& 999992.09 \\
\hline
{\small{r^2+{{r\left(\log(2)+\log(r+1)+r\log\left(\dfrac{r+1}{r}\right)\right)}}}} & 3.07944 & 9.20537 & 17.8274 & 28.7806 & 140.4414 & 10630.33 & 1008601.4\\\hline
Q(r) &11.8187& 2.725 & 1.5669 & 1.1179 & 0.4531 & 0.06362 & 0.0086\\ 
\hline
\end{array}$$}}}
\end{remark}}

\section{Examples of Theorem $\ref{Lerch}$}
\begin{example} 

Put $r=m=10$ and $x=0$. We consider $\boldsymbol{\alpha}:=(1,1/2,\ldots, 1/10)$, $K=\Q$ and $v_0=\infty$.
Let $b$ be an integer. {Then we have
\begin{align*}
V(\boldsymbol{\alpha},b)&=\log|b|-\left[100\log(2520)+100\log(2)+10\left(\log(101)+100\log\left(\dfrac{101}{100}\right)\right)\right]-1000\\
&>\log|b|-1908.6176\enspace.
\end{align*}}

{{Assume {{${\rm{log}}\,|b|>1908.6176$. }}
Then, by Theorem $\ref{Lerch 2}$, the $10^2+1$ numbers~$:$
$$1,{\rm{Li}}_{1}(1/b),\ldots,{\rm{Li}}_{10}(1/b),\ldots, {\rm{Li}}_{1}(1/(10b)),\ldots, {\rm{Li}}_{10}(1/(10b))\enspace,$$ are linearly independent over $\Q$.
}}

\end{example}
\begin{example} \label{ex 2}
Let $r,m$ be a positive integer. We consider $\boldsymbol{\alpha}:=(1,1/2,\ldots, 1/m)$.
Let $\beta$ be a root of $X^2-X-1$ and $x=0$. Put $K:=\Q(\beta)$. Let $v_0$ be the infinite place of $K$ satisfying $|\beta|_{v_0}>1$. 
In this case, we have 
${\rm{log}}|\beta|_{v_0}=\tfrac{1}{2}{\rm{log}}\left(\tfrac{1+\sqrt{5}}{2}\right).$
Let $M$ be a positive integer. Then we have
{\small{{\begin{align*}
V(\boldsymbol{\alpha},\beta^M)=\dfrac{M}{2}{\rm{log}}\left(\dfrac{1+\sqrt{5}}{2}\right)-rm\log\,(d_m)-\const-r^2m\enspace.
\end{align*}}}}
If $M$ is sufficiently large, the real number $V(\boldsymbol{\alpha},\beta^M)$ becomes positive and thus, by Theorem $\ref{Lerch 2}$, the $rm+1$ numbers~$:$
$$1,{\rm{Li}}_{1}(\alpha_1/\beta^M),\ldots,{\rm{Li}}_{r}(\alpha_1/\beta^M),\ldots, {\rm{Li}}_{1}(\alpha_m/\beta^M),\ldots, {\rm{Li}}_{r}(\alpha_m/\beta^M)\enspace,$$ are linearly independent over $K$.
For instance, we consider $r=m=10$ and $M\ge 7933$.  
Since we have ${\rm{log}}\left(\tfrac{1+\sqrt{5}}{2}\right)/2>0.2406059$, we obtain
\begin{align*}
V(\boldsymbol{\alpha},\beta^M)>0.2406059 \times M-1908.70272>0 \enspace.
\end{align*} 
Then the $10^2+1$ numbers,
$$1,{\rm{Li}}_{1}(1/\beta^M),\ldots,{\rm{Li}}_{10}(1/\beta^M),\ldots, {\rm{Li}}_{1}(1/(10\beta^M)),\ldots, {\rm{Li}}_{10}(1/(10\beta^M))\enspace,$$ 
are linearly independent over $\Q(\beta)$.
 
\bigskip

Assume $V(\boldsymbol{\alpha},\beta^M)>0$. Then the linear independence measure of $\{1,{\rm{Li}}_{s}(\alpha_i/\beta^M)\}_{\substack{1\le i \le m \\ 1\le s \le r}}$ is 
\begin{align*}
\mu(\boldsymbol{\alpha},\beta^M,\varepsilon)=\dfrac{\mathbb{A}(\boldsymbol{\alpha},\beta^M)+{{U}}(\boldsymbol{\alpha},\beta^M)}{V(\boldsymbol{\alpha},\beta^M)-\varepsilon}
=\dfrac{(rm+1)M{\rm{log}}\left(\tfrac{1+\sqrt{5}}{2}\right)/2}{V(\boldsymbol{\alpha},\beta^M)-\varepsilon}\enspace.
\end{align*} 
Thus we obtain $$\mu(\boldsymbol{\alpha},\beta^M,\varepsilon)\rightarrow rm+1 \ \ (M\to \infty)\enspace.$$ 
\end{example}
\begin{example}
Let $r,m$ be positive integers. We consider $\boldsymbol{\alpha}:=(1,1/2,\ldots, 1/m)$.
Let $d,M\in \N$ with $d,M\ge 3$. Define $f_{M,d}(X)\in \Q[X]$ by
\begin{align*}
&f_{M,d}(X):=\left(2+\dfrac{1}{M}\right)X^d-\dfrac{2}{M}X^{d-1}-2X+\dfrac{2}{M}\enspace.
\end{align*}
Let $1/\beta_M$ be a root of $f_{M,d}(X)$ and put $K=\Q(\beta_M)$.
Let $v_0$ be an infinite place of $K$ with the absolute value of ${\rm{log}}|1/\beta_M|_{v_0}$ takes the minimal value among $({\rm{log}}|1/\beta_M|_v)_{v\in {{\mathfrak{M}}}^{\infty}_{\Q(\beta_M)}}$.
Put $x=0$. 
Then we have
\begin{align*}
V(\boldsymbol{\alpha},\beta_M)&=\displaystyle  \log\vert\beta_M\vert_{v_0}+rm\log\Vert(\boldsymbol{\alpha},\beta_M)\Vert_{v_0}-rm{\mathrm{h}}(\boldsymbol{\alpha},\beta_M)\\
                                            &-\displaystyle\const-r^2m\enspace.
\end{align*}
Let $M_0$ be the minimal positive integer  with $|\beta^{(g)}_{M_0}|\le 2$ for $2\le g \le d$ and $M$ a positive integer with $M\ge M_0$.
Since ${\rm{den}}(\beta_M)\le 2$ and ${\rm{log}}|\beta_M|_v\le [K_v:\R]{\rm{log}}(2)/d$ for $v|\infty$, we have
\begin{align*}
V(\boldsymbol{\alpha},\beta_M)&\ge \dfrac{1}{d}{\rm{log}}|\beta_M|-rm\left(\dfrac{2d-1}{d}\log\,(2)+\log\,(d_m)\right)\\
&-\displaystyle\const-r^2m\enspace.
\end{align*}
The following tables are about ${\rm{log}}_{10}(M)$ with respect to $3\le d\le 5$ and $(r,m)$ satisfying the right hand side of the above inequality becomes positive. 
First we give the table in the case of $d=3$.
$$
\begin{array}{c|ccccccccc}
m \backslash r & 1 & 2 & 3 & 4 & 5 & 6 & 7 & 8 & 9 \\
\hline
 1 & 6 & 12 & 22 & 32 & 49 & 59 & 85 & 104 & 130 \\
 2 & 16 & 31 & 52 & 75 & 113 & 134 & 190 & 231 & 285 \\
 3 & 28 & 55 & 91 & 129 & 190 & 226 & 314 & 379 & 463 \\
 4 & 44 & 85 & 138 & 194 & 281 & 333 & 457 & 548 & 666 \\
 5 & 63 & 121 & 194 & 270 & 385 & 457 & 618 & 739 & 893 \\
 6 & 84 & 162 & 257 & 356 & 502 & 597 & 798 & 951 & 1143 \\
 7 & 109 & 208 & 329 & 454 & 632 & 752 & 996 & 1183 & 1417 \\
 8 & 136 & 260 & 408 & 561 & 776 & 923 & 1212 & 1437 & 1714 \\
 9 & 166 & 317 & 496 & 680 & 933 & 1110 & 1447 & 1712 & 2035 \\
\end{array}
$$
Second we give the table in the case of $d=4$.
$$
\begin{array}{c|ccccccccc}
 m \backslash r & 1 & 2 & 3 & 4 & 5 & 6 & 7 & 8 & 9 \\
 \hline
 1 & 8 & 16 & 29 & 42 & 66 & 78 & 115 & 140 & 174 \\
 2 & 21 & 41 & 70 & 100 & 151 & 180 & 255 & 309 & 381 \\
 3 & 38 & 74 & 122 & 173 & 255 & 302 & 421 & 508 & 620 \\
 4 & 59 & 114 & 185 & 260 & 376 & 447 & 612 & 734 & 892 \\
 5 & 84 & 162 & 260 & 362 & 515 & 612 & 827 & 989 & 1194 \\
 6 & 113 & 217 & 345 & 477 & 672 & 799 & 1068 & 1272 & 1529 \\
 7 & 145 & 279 & 440 & 607 & 847 & 1007 & 1332 & 1583 & 1895 \\
 8 & 182 & 348 & 547 & 752 & 1039 & 1236 & 1622 & 1922 & 2293 \\
 9 & 222 & 424 & 664 & 910 & 1248 & 1486 & 1936 & 2289 & 2722 \\
\end{array}
$$
Last we give the table in the case of $d=5$.
$$
\begin{array}{c|ccccccccc}
 m \backslash r & 1 & 2 & 3 & 4 & 5 & 6 & 7 & 8 & 9 \\
 \hline
 1 & 10 & 20 & 36 & 53 & 83 & 98 & 144 & 175 & 218 \\
 2 & 26 & 52 & 88 & 126 & 190 & 225 & 320 & 388 & 478 \\
 3 & 47 & 93 & 153 & 217 & 319 & 379 & 528 & 636 & 777 \\
 4 & 74 & 143 & 233 & 326 & 471 & 560 & 767 & 920 & 1117 \\
 5 & 105 & 203 & 326 & 454 & 646 & 767 & 1037 & 1239 & 1496 \\
 6 & 141 & 271 & 432 & 599 & 842 & 1001 & 1337 & 1594 & 1915 \\
 7 & 182 & 349 & 552 & 761 & 1061 & 1262 & 1669 & 1983 & 2373 \\
 8 & 228 & 436 & 685 & 942 & 1301 & 1548 & 2031 & 2408 & 2871 \\
 9 & 278 & 532 & 831 & 1140 & 1564 & 1861 & 2424 & 2867 & 3408 \\
\end{array}
$$
\end{example}

\bibliography{}

\begin{thebibliography}{99}%
%\bibitem{A-R}
%K.~Alladi and M.~L.~Robinson,
%\emph{Legendre polynomials and irrationality},
%J. Reine Angew Math., \textbf{318}, (1980), 137--155.

%\bibitem{An}
%Y.~Andr\'e, {\it $G$-functions and Geometry}, 
%Aspects of Mathematics, E\,13, Friedr. Vieweg \& Sohn, Braunschweig, 1989.


%\bibitem{A-B-V}
%A.~I.~Apetekarev, A.~Branquinho and W.~Van Assche,
%\emph{Multiple orthogonal polynomials for classical weights},
%Trans. Amer. Math. Soc., \textbf{355},\ no.\ 10,  (2003),  3887--3914.

%\bibitem{Baker1975}
%A.~Baker, \emph{Transcendental Number Theory},  
%Cambridge Univ. Press, 1975.

%\bibitem{Bel}
%P.~Bel,  Fonctions $L ~p$-adiques et irrationalit\'e, 
%Ann. Scuola Norm. Sup. Pisa Cl. Sci., (5), \textbf{9}, (2010), no. 1, 189--227.

%\bibitem{beu}
%F.~Beukers,
%\emph{A note on the irrationality of  $\zeta(2)$ and $\zeta(3)$},
%Bull. London Math. Soc., \textbf{11}, (1979),  268--272.

\bibitem{Bel} 
P.~Bel,
{\it Fonctions $L$ $p$-adiques et irrationalit\'e},  Ann. Sc. Norm. Super. Pisa Cl. Sci. \textbf{9} (2010), 189--227.

\bibitem{B}
F.~Beukers,
{\it Irrationality of some $p$-adic $L$-values}, Acta Math. Sin. (Engl. Ser.) 
%Acta Math. Sin. \textbf{24} %no.\ 4, 
(2008), 663--686.

\bibitem{bom} 
E.~Bombieri, 
{\it On the Thue-Siegel-Dyson Theorem}, 
Acta Math. \textbf{148} (1982), 255--296.

\bibitem{bomhunpoor} 
E.~Bombieri, D.~C.~Hunt and A.~J.~Van der Poorten, 
{\it Determinants in the Study of Thue's Method and Curves with Prescribed Singularities}, Experiment. Math. \textbf{4} (1995), 87--96.

\bibitem{bomcoza} 
E.~Bombieri, P.~Cohen,
{\it Siegel's Lemma, Pad\'e Approximations and jacobians}, 
with an appendix by U.~Zannier, 
Ann. Scuola Norm. Sup. Pisa Cl. Sci. \textbf{25} (1997), 155--178.

%\bibitem{ch2}
%G.~V.~Chudnovsky,
%\emph{Pad\'e approximations to the generalized hypergeometric functions I},
%J. Math. Pures et Appl.,  \textbf{58}, (1979),  445--476.

%\bibitem{ch3}
%G.~V.~Chudnovsky,
%\emph{Measures of irrationality, transcendence and algebraic independence, Recent progress},
%London Math. Soc. Lecture Notes ser.,  \textbf{56}, Cambridge Univ. Press, (1982),  11--82.

%\bibitem{ch9}
%G.~V.~Chudnovsky,
%\emph{On the method of Thue-Siegel},
%Annals of Math., \textbf{117}, (1983),  325--382.

\bibitem{ch11}
G.~V.~Chudnovsky,
{\it On applications of Diophantine approximations},
Proc. Nat. Acad. Sci. U.S.A. \textbf{81} (1984), 1926--1930.

\bibitem{Chubrothers}
D.~V.~Chudnovsky and G.~V.~Chudnovsky,
{\it Applications of Pad\'e approximations to diophantine inequalities in values of $G$-functions},
In: Number theory (New York, 1983--84), D.~V.~Chudnovsky, G.~V.~Chudnovsky, H.~Cohn,  M.~B.~Nathanson (eds.),
Lecture Notes in Math. \textbf{1135} (1985), 9--51.

%\bibitem{ch10}
%G.~V.~Chudnovsky,
%\emph{Use of computer algebra for diophantine and differential equations},
%in Computer Algebra, M.~Dekker, NY, (1988), 1--82.

\bibitem{C-M}
B.~\'{C}urgus, V.~Mascioni,
{\it Roots and polynomials as Homeomorphic space},
Expo. Math. \textbf{24} (2006), 81--95.

\bibitem{DHK2}
S.~David, N.~Hirata-Kohno  and M.~Kawashima,
{\it Can polylogarithms at algebraic points be linearly independent?}, Mosc. J. Comb. Number Theory \textbf{9} (2020), 389--406.

%\bibitem{DHK3}
%S.~David, N.~Hirata-Kohno  and M.~Kawashima,
%{\it Linear Forms in Polylogarithms},
%arxiv address of this article:

\bibitem{DHK4}
S.~David, N.~Hirata-Kohno and M.~Kawashima,
{\it Linear independence criterion of the Lerch functions with distinct shifts}, preprint.

\bibitem{Fli}  
Y.~Z.~Flicker,
{\it On $p$-adic $G$-functions},
J. London Math. Soc. (2) \textbf{15} (1977), 395--402.

\bibitem{FSZ}
S.~Fischler, J.~Sprang and W.~Zudilin,
{\it Many odd zeta values are irrational}, 
Compos. Math.  \textbf{155} (2019), 938--952. 

\bibitem{G1}
A.~I.~Galochkin,
{\it Lower bounds of polynomials in the values of a certain class of analytic functions},
Mat. Sb. (N.~S.) \textbf{95~(137)} (1974), 396--417, 471;  English translation in Math. USSR-Sb. \textbf{24}, no.~3 (1974).

\bibitem{G2}
A.~I.~Galochkin,
{\it Lower bounds of linear forms of the values of certain $G$-functions},
Mat. Zametki \textbf{18}, no.~4 (1975) 541--552; English transl. in Math. Note \textbf{18}, (1975).


\bibitem{G3}
A.~I.~Galochkin,
{\it Criterion for membership of hypergeometric Siegel functions in a class of $E$-functions},
Mat. Zametki \textbf{29}, no.~1 (1981), 3--14, 154; English translation in Math. Note \textbf{29} (1981), 3--8.


\bibitem{Ha}
M.~Hata,
{\it On the linear independence of the values of polylogarithmic functions},
J. Math. Pures et Appl. \textbf{69} (1990), 133--173.

\bibitem{Ha1993}
M.~Hata,
\emph{Rational approximations to the dilogarithms},
Trans. Amer. Math. Soc. \textbf{336} (1993), 363--387.

\bibitem{H-I-W}
N.~Hirata-Kohno, M.~Ito and Y.~Washio,
{\it A criterion for the linear independence of polylogarithms over a number field}, In: Algebraic number theory and related topics 2014,
Res. Inst. Math. Sci. (RIMS), Kyoto, 
RIMS K\^oky\^uroku Bessatsu \textbf{64} (2017), 3--18.

\bibitem{H-K-S}
M.~Hirose, M.~Kawashima and N.~Sato,
{\it A lower bound of the dimension of the vector space spanned by the special values of certain functions},
Tokyo J. Math. \textbf{40} (2017), 439--479.

\bibitem{Ivan} P.~L.~Ivankov,
{\it Arithmetic properties of values of hypergeometric functions}, Mat. Sb. \textbf{182} (1991), 283--302;  English translation in 
Math. USSR-Sb. \textbf{72} (1992), 267--286.

%\bibitem{Katz} 
%N.~Katz,
%{\it Nilpotent connections and the monodoromy theorem: application of a result of Turritin}, 
%Publ. Math. IHES, \textbf{39},  (1970),  175--232.

\bibitem{Ka}
M.~Kawashima,
{\it Evaluation of the dimension of the $\Q$-vector space spanned by the special values of the Lerch function},
 Tsukuba J. Math. \textbf{38} (2014), 171--188.

%\bibitem{Lang}
%S.~Lang,
%{\it Fundamentals of Diophantine Geometry},
%Springer, 1983, ISBN 3-540-90837-4.

\bibitem{marc}
R.~Marcovecchio,
{\it Linear independence of forms in polylogarithms},
 Ann. Sc. Norm. Super. Pisa Cl. Sci. \textbf{5}  (2006), 1--11.

\bibitem{M}
M.~Marden,
{\it Geometry of polynomials, Second edition, reprinted with corrections},
Mathematical Surveys, \textbf{3}, American Mathematical Society, Providence, R.I. 1966.


\bibitem{Mi} 
M.~A.~Miladi,
{\it R\'ecurrences lin\'eaires et approximations simultan\'ees de type Pad\'e: applications \`a l'arith\-m\'etique},
Th\`ese, Universit\'e des S. et T. de Lille, 2001.

\bibitem{Nest} 
Yu.~Nesterenko, {\it Hermite-Pad\'e approximants of generalized hypergeometric functions},Mat. Sb. \textbf{185} no.~10 (1994),
39--72;  English translation in 
Russian Acad. Sci. Sb. Math. \textbf{83} (1995), 189--219.


\bibitem{N}
E.~M.~Niki\v{s}in,
{\it Irrationality of values of functions $F(x,\,s)$}, Mat. Sb. (N.~S.) \textbf{109 (151)}, no.~3 (1979), 410--417, 479; 
English translation in Math. USSR-Sb. \textbf{37} (1980), 381--388.

%\bibitem{N-S}
%E.~M.~Nikisin and V.~N.~Sorokin,
%{\it Rational Approximations and Orthogonality}, Translations of Mathematical Monographs,
%American Math. Society, 1991. 

\bibitem{Pade1} 
H.~Pad\'e,
{\it Sur la repr\'esentation approch\'ee d'une fonction par des fractions rationnelles},
Ann. Sci. \'Ecole Norm. Sup. \textbf{9} (1892), 3--93.

\bibitem{Pade2}
H.~Pad\'e, 
{\it M\'emoire sur les d\'eveloppements en fractions continues de la fonction exponentielle, pouvant servir d'introduction \`a la th\'eorie des fractions continues alg\'ebriques}, 
Ann. Sci. \'Ecole Norm. Sup. \textbf{16} (1899), 395--426. 

\bibitem{R-T}
G.~Rhin and P.~Toffin,
{\it Approximants de Pad\'{e} simultan\'{e}s de logarithmes},
J. Number Theory, \textbf{24} (1986), 284--297.

%\%bibitem{RV0}
%G.~Rhin and  C.~Viola,
%On a permutation group related to $\zeta (2)$, Acta Arith., \textbf{77},
%(1996),  \ no.\ 1,  23--56.

\bibitem{RV1}
G.~Rhin and  C.~Viola,
{\it The permutation group method for the dilogarithms},
Ann. Sc. Norm. Super. Pisa Cl. Sci. \textbf{4} (2005), 389--437.

%\bibitem{Rizeta}
%T. Rivoal,
%{\it 
%Irrationalit\'e d'au moins un des neuf nombres} $\zeta(5), \zeta(7), \ldots, \zeta(21)$, Acta Arith.,
%\textbf{103}, \ no.\ 2,  (2002), 157---167.

\bibitem{Ri}
T.~Rivoal,
{\it Ind\'ependance lin\'eaire des valeurs des polylogarithmes},
J. Th\'eor. Nombres Bordeaux, \textbf{15} (2003), 551--559.

%\bibitem{R-S1}
%J.~B.~Rosser and L.~Schoenfeld,
%{\it Approximate formulas for some functions of prime numbers},
%Illinois J.\ Math., \textbf{6}, (1962), 64--94.

%\bibitem{R-S2}
%J.~B.~Rosser and L.~Schoenfeld,
%{\it Shaper bounds for the Chebyshev functions $\theta(x)$ and $\psi(x)$},
%Math.\ Comp., \textbf{29}, (1975), 243--269.

\bibitem{Rilerch}
T.~Rivoal,
{\it Simultaneous Polynomial Approximations of the Lerch Function},
Canad. J. Math. \textbf{61} (2009), 1341--1356.


\bibitem{R}
T.~Rivoal,
{\it On Galochkin's characterization of hypergeometric $G$-functions}, 
Mosc. J. Comb. Number Theory  \textbf{11~(1)} (2022), 11-19.


\bibitem{Siegel} 
C.~L.~Siegel,
{\it \"Uber einige Anwendungen diophantischer Approximationen}, 
Abhandlungen der Preu\ss{}ischen Akademie der Wissenschaften.
Physikalisch-Mathematische Kl. (1929--30), 1--70.

\bibitem{SprangL}
J.~Sprang,
{\it A Linear independence result for $p$-adic $L$-values}, Duke Math. J. \textbf{169} (2020), 3439--3476.
%https://arxiv.org/abs/1809.07714.

%\bibitem{S}
%G.~Szeg$\ddot{\text{o}}$,
%{\it Orthogonal Polynomials}, American Math. Society, 1939.

\bibitem{Va}
K.~V$\ddot{\text{a}}$$\ddot{\text{a}}$n$\ddot{\text{a}}$nen,
{\it On linear forms of a certain class of $G$-functions},
Acta Arith. \textbf{36} (1980), 273--295.

%\bibitem{Va-Gu}
%K.~V$\ddot{\text{a}}$$\ddot{\text{a}}$n$\ddot{\text{a}}$nen, G. Xu
%{\it On linear forms of $G$-functions},
%Acta Arith., Vol. \textbf{50}, (1988), 251--263.

\bibitem{VZ1}
C.~Viola and W.~Zudilin,
{\it Linear independence of dilogarithmic values},
J. Reine Angew. Math. \textbf{736} (2018), 193--223.

\bibitem{Za}
U.~Zannier,
{\it Hyperelliptic continued fractions and generalized Jacobians},
 Amer. J. Math. \textbf{141} (2019), 1--40.

\bibitem{Z}
W.~Zudilin,
{\it On a measure of irrationality for values of $G$-functions},  Izv. Ross. Akad. Nauk Ser. Mat. \textbf{60} (1996), 
87--114;  English translation in  Izv. Math.  \textbf{60} (1996), 91--118.
\end{thebibliography}

\begin{scriptsize}
\begin{minipage}[t]{0.35\textwidth}

Sinnou David,
\\Institut de Math\'ematiques
\\de Jussieu-Paris Rive Gauche
\\CNRS UMR 7586,
Sorbonne Universit\'{e}
\\
4, place Jussieu, 75005 Paris, France\\
\& CNRS UMI 2000 Relax\\
 Chennai Mathematical Institute\\
 H1, SIPCOT IT Park, Siruseri\\
 Kelambakkam 603103, India \\
sinnou.david@imj-prg.fr
%david@imj-prg.fr
\\\\
\end{minipage}
\begin{minipage}[t]{0.31\textwidth}
Noriko Hirata-Kohno, 
\\Department of Mathematics
\\College of Science \& Technology
\\Nihon University
\\Kanda, Chiyoda, Tokyo
\\101-8308, Japan
\\hirata@math.cst.nihon-u.ac.jp\\\\
\end{minipage}
\begin{minipage}[t]{0.29\textwidth}  
Makoto Kawashima,
\\Department of Liberal Arts \\and Basic Sciences
\\College of Industrial Engineering
\\Nihon University
\\Izumi-chou, Narashino, Chiba
\\275-8575, Japan\\
kawashima.makoto@nihon-u.ac.jp\\\\
\end{minipage}

\end{scriptsize}

\end{document}